\documentclass{IEEEtran}
\usepackage{amsmath}
\usepackage{cite}
\usepackage{amsthm}
\usepackage{amssymb}
\usepackage{tikz}
\usetikzlibrary{shapes,arrows}
\usepackage{subcaption}
\usepackage{capt-of}
\usepackage[labelformat=simple]{subcaption}

\usepackage{enumerate}
\usepackage{epstopdf}
\usepackage{pgfplots}
\usepackage{enumerate}
\usepackage{hyperref}
\usepackage{microtype}
\usepackage{graphicx}
\usepackage{booktabs} 

\newtheorem{defin}{Definition}
\newtheorem{theorem}{Theorem}
\newtheorem{prop}{Proposition}
\newtheorem{lemma}{Lemma}
\newtheorem{corollary}{Corollary}
\newtheorem{remark}{Remark}

\newtheorem{assumption}{Assumption}

\newcommand{\be}{\begin{equation}}
\newcommand{\ee}{\end{equation}}
\newcommand{\ba}{\begin{array}}
\newcommand{\ea}{\end{array}}
\newcommand{\bea}{\begin{eqnarray}}
\newcommand{\eea}{\end{eqnarray}}

\newcommand{\tran}{^{\mbox{\scriptsize T}}}  

\newcommand{\vbar}{\raisebox{.17ex}{\rule{.04em}{1.35ex}}}
\newcommand{\vbarind}{\raisebox{.01ex}{\rule{.04em}{1.1ex}}}
\newcommand{\D}{\ifmmode {\rm I}\hspace{-.2em}{\rm D} \else ${\rm I}\hspace{-.2em}{\rm D}$ \fi}
\newcommand{\T}{\ifmmode {\rm I}\hspace{-.2em}{\rm T} \else ${\rm I}\hspace{-.2em}{\rm T}$ \fi}
\newcommand{\B}{\ifmmode {\rm I}\hspace{-.2em}{\rm B} \else \mbox{${\rm I}\hspace{-.2em}{\rm B}$} \fi}
\newcommand{\Hil}{\ifmmode {\rm I}\hspace{-.2em}{\rm H} \else \mbox{${\rm I}\hspace{-.2em}{\rm H}$} \fi}
\newcommand{\C}{\ifmmode \hspace{.2em}\vbar\hspace{-.31em}{\rm C} \else \mbox{$\hspace{.2em}\vbar\hspace{-.31em}{\rm C}$} \fi}
\newcommand{\Cind}{\ifmmode \hspace{.2em}\vbarind\hspace{-.25em}{\rm C} \else \mbox{$\hspace{.2em}\vbarind\hspace{-.25em}{\rm C}$} \fi}
\newcommand{\Q}{\ifmmode \hspace{.2em}\vbar\hspace{-.31em}{\rm Q} \else \mbox{$\hspace{.2em}\vbar\hspace{-.31em}{\rm Q}$} \fi}
\newcommand{\Z}{\ifmmode {\rm Z}\hspace{-.28em}{\rm Z} \else ${\rm Z}\hspace{-.38em}{\rm Z}$ \fi}

\renewcommand{\vec}[1]{{\bf{#1}}}     

\def\etal{\textit{et. al.}}

\newcommand{\R}{\mathbb{R}}
\newcommand{\N}{\mathbb{N}}
\newcommand{\bec}[1]{\bar{\vec{#1}}}

\newcommand{\BLa}{\boldsymbol\Lambda}

\usepackage{float}

\usepackage{hyperref}

\begin{document}

\title{On Maintaining Linear Convergence of Distributed Learning  and Optimization under Limited Communication}

\author{
  Sindri Magn\'usson, 
   Hossein Shokri-Ghadikolaei, and
   Na Li
   
\thanks{*This  work  was  supported  in  part by the Wallenberg AI, Autonomous Systems and Software Program (WASP) funded by the Knut and Alice Wallenberg Foundation, the Swedish Foundation for Strategic Research (SSF), the Swedish Research Foundation under grant 2018-00820, NSF CAREER: ECCS-1553407, AFOSR YIP: FA9550-18-1-0150, and ONR YIP: N00014-19-1-2217.}
\thanks{Sindri Magn\'usson is with the Department of Computer and Systems Science, Stockholm University and Sweden and KTH Royal Institute of Technology, Stockholm, Sweden (e-mail: sindri.magnusson@dsv.su.se). 

Hossein Shokri-Ghadikolaei is with EPFL, Laussane , Switzerland, and KTH Royal Institute of Technology, Stockholm, Sweden.  (e-mail: hshokri@kth.se)

Na Li is with the Harvard John A. Paulson School of Engineering and Applied Science  (e-mail:  nali@seas.harvard.edu)

}

%
}

%






\maketitle

\begin{abstract}

In distributed optimization and machine learning, multiple nodes coordinate to solve large problems. To do this, the nodes need to compress important algorithm information to bits so that it can be communicated over a digital channel. The communication time of these algorithms follows a complex interplay between a) the algorithm's convergence properties, b) the compression scheme, and c) the transmission rate offered by the digital channel. We explore these relationships for a general class of linearly convergent distributed algorithms.   In particular,  we illustrate how to design quantizers for these algorithms that compress the communicated information to a few bits while still preserving the linear convergence. Moreover, we characterize  the communication time of these algorithms as a function of the available  transmission rate. We illustrate our results on learning algorithms using different communication structures, such as decentralized algorithms where a single master coordinates information from many workers and fully distributed algorithms where only neighbours in a communication graph can communicate. We conclude that a co-design of machine learning and communication protocols are mandatory to flourish machine learning over networks.

\end{abstract}

\section{Introduction}

Large-scale distributed computing systems are the cornerstone of recent advancements in many disciplines such as machine learning (ML), communication networks, and networked control. For example, efficient parallel processing has enabled analysis and optimization over big datasets. Similarly, spatially separated wireless networks with cheap and low-complexity sensor nodes (e.g., Internet of things, smart grids, and intra-body wireless sensor networks) are revolutionizing our infrastructures and societies. 
In these systems, multiple processors coordinate to solve large and often spatially separated computational problems. The computational burden of the individual processors is usually manageable, if not small, and based on elementary operations. Instead, the main complexity often lies in the coordination and communication among the processors. This problem is exacerbated in wireless networks where bandwidth-limited and faulty wireless links may further complicate the coordination and consequently become a bottleneck of the distributed computing system. One may implement a set of communication techniques (e.g., quantization, coding and modulation, and scheduling) to reliably exchange information bits among processors at the expense of a slower communication rate.

In an attempt to quantify this coordination challenge in distributed computing,~\cite{yao1979some} introduced the notion of communication complexity in 1970s. 
It measures the minimal number of transmitted bits required between multiple processors to compute a binary function whose inputs are distributed among them (in the minimax sense). 
More recent works have investigated communication complexity and communication-efficient algorithms in various systems, such as in networked control~\cite{hespanha2007survey}, distributed  optimization~\cite{tsitsiklis1987communication,rabbat2005quantized,nedic2008distributed,magnusson2017convergence,magnusson2018communication,WuError2018,YeCommunication2018,alistarh2018convergence,khirirat2018distributed},  and equilibrium seeking in games~\cite{conitzer2004communication,hart2010long,nekouei2016performance}.
 Communication-efficiency has also gained a massive recent interest in the ML community, where  parallel and distributed algorithms are becoming increasingly important in dealing with the huge data size~\cite{zhang2013information,alistarh2017qsgd,ZhuDistributed2018,jordan2018communication,balcan2012distributed}. 
 In fact, when training many of the state-of-the art deep neural networks the communication time is starting to outweigh the computing time~\cite{alistarh2017qsgd}. 
 %
 %
%
A main approach to address the communication complexity is to reduce the number of bits to represent the exchanged information vectors. 
Recent works have analyzed the possibility of updating based on the quantized gradients, both in deterministic~\cite{jordan2018communication,nekouei2016performance,magnusson2019maintaining} and stochastic~\cite{bernstein2018signsgd,kamilov2018signprox,wen2017terngrad,seide20141} settings.

The conventional wisdom, validated by empirical observations, is that there exists a precision-accuracy tradeoff: the fewer the number of bits the lower the accuracy of the final solution.
Some recent studies, however, challenged that wisdom in the deterministic~\cite{lee2018finite,pu2016quantization,magnusson2019maintaining} and stochastic~\cite{de2018high,Stich2018Sparsified} settings. 
 In particular, these papers illustrate how the convergence rate can be preserved under limited data-rates by using adaptive quantization scheme that shrinks as the algorithm converges.
 This idea is, to the best of our knowledge, first presented in~\cite{tsitsiklis1987communication}, which provides almost tight upper and lower bounds on the number of bits two nodes need to communicate to approximately solve a strongly convex optimization problem.
 All of the above adaptive  quantization schemes are, however, designed  for particular algorithms.
It is one of the goals of this paper to illustrate how such adaptive quantization scheme can be designed for general distributed algorithms that converges linearly in any norm.
%
%
%

In the second part of this paper, we highlight that the design of the existing distributed optimization algorithms are often ignorant to some important objectives including end-to-end latency, required to solve an optimization problem. 
These factors become of paramount importance when we implement ML and distributed optimization algorithms over a network to support low-latency services, such as industrial automation and online gaming. We consider a set of mappings between some design parameters (the packet size, the reliability of the communication channel, and multiple access protocol) and the achievable transmission rate. We then characterize the aggregated latency required for the convergence of the distributed algorithm. We show a surprising result that reducing the number of bits per iteration (even if we can maintain the linear convergence rate) may not necessarily lead to the lowest latency for solving an optimization problem. In particular, in a contention-based resource allocation such as slotted-ALOHA~\cite{bertsekas1992data}, a high-dimensional ML problem may enjoy sending the fewest number of bits per iteration (to control the channel contention at an optimal level) even though it may not lead to the optimal convergence rate in terms of the number of iterations, whereas a high rate point-to-point network may enjoy a much higher value for the optimal number of bits.

The main contributions of this paper are as follows.
\begin{itemize}
\item \emph{Iteration complexity analysis:} We considers general parallel/distributed algorithms that have a linear convergence rate in an arbitrary norm. For these algorithms, we develop an adaptive quantization that maintains the non-asymptotic and asymptotic linear convergence rate while communicating only few bits per iteration. 
\item \emph{Novel performance measure:} In the second part of the paper, we introduce the novel notion of transmission time complexity of running a distributed optimization over a generic communication network. We consider abstract transmission rate functions that take as inputs the size of the payload (number of quantization bits) and overhead of each packet and the probability of package failure and quantify. 
\item \emph{Transmission time complexity analysis:} We build on our iteration-complexity results and characterize  the transmission time convergence of distributed algorithms for different communication channels, such as  AWGN channels and  multiple access channels. Our algorithm class/framework is generic enough to cover many ML algorithms of practical interest and with   different communication protocols.
\item \emph{Example use cases:} We illustrate the usage of our algorithm to maintain convergence of distributed ML and optimization algorithms where i) a master node coordinates information received from many worker nodes and ii) where the nodes communicate over a communication graph with no central coordinator.
\item \emph{Extensive experiments} We show the implications of our theoretical convergence results on large dimensional data sets, including the MNIST.
\end{itemize}
This paper addresses the problem of communication complexity from the point of view of information and communication theory, which is still in its infancy.
 A conference version containing part of this work  was presented in~\cite{magnusson2019maintaining}.
 All of the work in Section~\ref{sec:Comparison},~\ref{Sec:TTC}, and~\ref{section:InterPlay} 
  is appearing here for the first time.
 Moreover, the discussions and results in Section~\ref{sec:AMCLC} and Section~\ref{Sec:Maintain} have been largely improved. 
 Moreover, the conference version did not include any proofs and almost all of the numerical results are new.

\textbf{Notations.} Normal font small letters $x$, bold font small letters $\vec{x}$, bold font capital letters $\vec{X}$, and calligraphic font $\mathcal{X}$ stand for scalars, vectors, matrices, and sets respectively. For a matrix $\vec{X}$, $\vec{X}_{ij}$ denotes its $(i,j)$-th entry. 
$\vec{A}\otimes \vec{B}$ is the Kronecker product $\vec{A}$ and $\vec{B}$ and $\texttt{Im}(\vec{A})$ is the span of $\vec{A}$. 
We denote by $\lambda_{\min}^+(\vec{A})$ and  $\lambda_{\max}(\vec{A})$, respectively, the smallest non-zero eigenvalue of $\vec{A}$ and the largest eigenvalue of matrix $\vec{A}$. We denote by $|\mathcal{L}|$  the cardinality of the set $\mathcal{L}$.  

\section{Algorithm Model: \\Communication and Linear Convergence} \label{sec:AMCLC}
In this section, we start by introducing the abstract form of our distributed algorithms. We then provide two examples of algorithms in this form that are popular in distributed learning over networks.

\subsection{Distributed Iterative Algorithms} \label{Sec:CIA}

Consider a network of $N$ nodes that  cooperatively solve a distributed computational problem involving some communication. In particular, we consider the following general algorithm framework
\begin{subequations} \label{eq:ItAlg1}
\begin{align}
     \vec{x}^{k+1}=& A(\vec{c}^k,    \vec{x}^k), \label{subeq:ItAlg1a}\\
    \vec{c}_i^{k+1}=& C_i(\vec{x}^{k+1}), ~~\text{ for } i=1,\ldots,N, \label{subeq:ItAlg1b}
 \end{align}
 \end{subequations}
where $\vec{c}^k=(\vec{c}_1^k,\ldots,\vec{c}_N^k)$.
The function   $A:\R^{Nd}\times \mathcal{X}\rightarrow \mathcal{X}$ represents an algorithm update of the  decision variable  $\vec{x}\in \mathcal{X}$, where $\mathcal{X}$ is a subset of a finite dimensional Euclidean space.
 The function $C_i:\mathcal{X}\rightarrow \R^d$ picks out the relevant information $\vec{c}_i=C_i(\vec{x})$ that node $i$ needs to communicate to run the algorithm.\footnote{To simplify the presentation we have assumed that $\vec{c}_i$ has the same dimension for all $i$. However, all the results in this paper also hold for $\vec{c}_i^k\in\R^{d_i}$ by replacing $d$ with $\max_{i} d_i$ followed by a zero padding.}
  This general algorithmic framework covers many ML algorithms. One example is that a master server performs the algorithm update based on information computed and communicated from many servers; see the first example in Section~\ref{section:Application}.
   The framework in Eq.~\eqref{eq:ItAlg1} also covers distributed learning algorithms where the nodes communicate over a communication graph $(\mathcal{N},\mathcal{E})$ where $\mathcal{N}:=\{1,2,\cdots,N\}$ stands for the nodes and $\mathcal{E}\subset \mathcal{N}\times \mathcal{N}$ stands for the communication links.
    This can be captured by the following iterations 
\begin{equation}
\begin{aligned} \label{eq:ItAlg2}
     \vec{x}_i^{k+1}=& A_i(\vec{c}_{\mathcal{N}_i}^{k},    \vec{x}_i^k),\\
    \vec{c}_i^{k+1}=& C_i(\vec{x}_i^{k+1}),
 \end{aligned}
 \end{equation}
where $\mathcal{N}_i:=\{j\in \mathcal{N} : (i,j)\in \mathcal{E}\}$ denotes the set of neighbors of node $i$, the function
 $A_i:\R^{d}\times \mathcal{X}_i\rightarrow \mathcal{X}_i$ is  the local algorithm update at node $i$, and  $C_i:\mathcal{X}_i\rightarrow \R^{d}$ is the information that node $i$ communicates to its neighbors. To express the algorithm in the form of Eq.~\eqref{eq:ItAlg1} we set
$$\vec{x}^k=(\vec{x}_1^k,\ldots,\vec{x}_N^k)~~\text{ and }~~\vec{c}^k=(\vec{c}_1^k,\ldots,\vec{c}_N^k),$$
and define the function  $A:\R^{Nd}\times \mathcal{X}\rightarrow \mathcal{X}$, $\mathcal{X}=\prod_{i=1}^N \mathcal{X}_i$, resulting in
$$A(\vec{c},\vec{x})=(A_1(\vec{c}_{\mathcal{N}_1},\vec{x}_1),\ldots,A_N(\vec{c}_{\mathcal{N}_N},\vec{x}_N)).$$
The focus of this paper are algorithms in the form of Eq.~\eqref{eq:ItAlg1} that have linear convergence rates, which we define as follows.
 \begin{defin} \label{Defin:LinearConvergence}
   We say that the algorithm in the form of Eq.~\eqref{eq:ItAlg1} is $\sigma$-linear convergent in the norm $||\cdot||$ if:  
    \begin{enumerate}[a)]
        \item The function
         $\vec{x} \mapsto A(C(\vec{x}), \vec{x})$ is $\sigma$-pseudo contractive on $\mathcal{X}$, i.e., there exists $\vec{x}^{\star}\in \mathcal{X}$ such that
        $$||A(C(\vec{x}), \vec{x})-\vec{x}^{\star}||\leq \sigma ||\vec{x}-\vec{x}^{\star}||,$$
         for all $\vec{x}\in \mathcal{X}$.
       \item         There exist $L_C$ and $L_A$ such that  
 \begin{subequations} \label{eq:AlgLips}
\begin{align}
            ||A(\vec{c}_1,\vec{x})-A(\vec{c}_2,\vec{x})||\leq& L_A ||\vec{c}_1-\vec{c}_2||_{\infty} ,
             \label{eq:AlgLips-1} \\
             ||C(\vec{x}_1)-C(\vec{x}_2)||_{\infty}\leq& L_C ||\vec{x}_1-\vec{x}_2||,
              \label{eq:AlgLips-2}
        \end{align}
 \end{subequations}
        for all $\vec{x}_1,\vec{x}_2,\vec{x}\in \mathcal{X}$ and $\vec{c}_1,\vec{c}_2\in\R^{Nd}$.
    \end{enumerate}
 \end{defin}
Note that this definition implies that the algorithm converges linearly in both $\vec{c}$ and $\vec{x}$ to some fixed points $\vec{x}^{\star}\in \mathcal{X}$ and $\vec{c}^{\star}\in \R^{Nd}$. In particular, we have
\begin{align*}
    ||\vec{x}^k-\vec{x}^{\star}||\leq& \sigma^k ||\vec{x}^0-\vec{x}^{\star}|| ~~~~~~~~\text{ and }\\
    ||\vec{c}^k-\vec{c}^{\star}||_{\infty}\leq& L_C \sigma^k ||\vec{x}^0-\vec{x}^{\star}||~~~~\text{ for all }~k\in \N.
\end{align*}
Definition~\ref{Defin:LinearConvergence}-b) is a smoothness condition on the mapping between  the algorithm iterations $\vec{x}$ and the communication variables $\vec{c}$. This is needed when considering quantized communication because quantization creates a noise that affects the algorithm performance. We use $L_{\infty}$-norm to measure the distance between communication variables because  our quantization (given in Section~\ref{Sec:Maintain}) is  a rectangular grid where distances are measured in the $L_{\infty}$-norm (Manhattan distance).

Many parallel and distributed algorithms can be expressed in this  form and have $\sigma$-linear convergence under certain conditions~\cite{bertsekas1989parallel}.
 For example, algorithms for learning Equilibriums in games~\cite{cui2008game,chen2010random,nekouei2016performance},
  coordination algorithms in communication networks~\cite{yates1995framework,Fischione_2011,Jakobsson_2016},
 and distributed optimization algorithms~\cite{shi2014linear,qu2018harnessing,uribe2018dual}. Below, we give two concrete examples of distributed learning algorithms in this form. 


\subsection{Application: Distributed Learning and Optimization} \label{section:Application}
 Consider the following optimization  problem
\begin{equation} \label{eq:mainOptProb}
\begin{aligned}
& \underset{\vec{z}\in \R^d}{\text{minimize}}
& &  F(\vec{z})= \sum_{i=1}^N f_i(\vec{z})
\end{aligned}
\end{equation}
 where $N$ nodes wish to learn a parameter $\vec{z}\in \R^d$ by minimizing the sum of local loss functions $f_i:\R^d\rightarrow \R$ based on the private data locally available at node $i$.
 We make the following assumption.
 \begin{assumption} \label{assumption:smooth}
   For $i{=}1,{\ldots},N$,  $f_i(\cdot)$ is $\mu$-strongly convex and has $L$-Lipschitz continuous gradient. 
 \end{assumption}

Below, we illustrate two typical algorithms, covered by our model in Eq.~\eqref{eq:ItAlg1}, that use different communication structures.

   \subsubsection{Decentralized Learning}  \label{sec:ExampleDL} 
 This problem is typically solved using the decentralized gradient update 
  $$\vec{x}^{k+1}=\vec{x}^k-\gamma\sum_{i=1}^N \nabla f_i(\vec{x}^k).$$
 To perform this update some communication is needed. The most common communication protocols are a) the nodes broadcast their gradients, then the nodes perform the gradient update locally, b) the nodes communicate their gradients to a master node that performs the gradient update.
   In either case,  this  algorithm is captured by the model in Eq.~\eqref{eq:ItAlg1} as follows
   \begin{equation}
   \label{eq:DL_star}
   \begin{aligned}
    \vec{x}^{k+1}=&A(\vec{c}^k,\vec{x}^k)= \vec{x}^k- \gamma\sum_{i=1}^N \vec{c}_i^k\  \\
    \vec{c}_i^{k+1}=& C_i(\vec{x}^{k+1})=  \nabla f_i(\vec{x}^{k+1}) ~\text{for } i=1,\ldots,N.
  \end{aligned}
  \end{equation}
  This algorithm is linearly convergent,
  following the standard analysis of gradient descent. 
  \begin{prop}[\cite{nesterov2013introductory}] \label{Theorem:Grad1}
   Let $A(\cdot)$ and $C(\cdot)$ be the functions defined in Eq.~\eqref{eq:DL_star} and
    $\vec{x}^{\star}$ be the optimal solution to the optimization problem in Eq.~\eqref{eq:mainOptProb}.
  Then
    $$||A(C(\vec{x}), \vec{x})-\vec{x}^{\star}||_2\leq \sigma ||\vec{x}-\vec{x}^{\star}||_2,~~\text{ for all }~\vec{x}\in \R^d,$$
where $\sigma\in [0,1)$ if $\gamma$ is small enough. For example,   $\sigma = 1-2/(\kappa+1)$ if $\gamma=2/[N(\mu+L)]$, where $\kappa=L/\mu$.
  It can also be verified that 
  Eq.~\eqref{eq:AlgLips-1} and Eq.~\eqref{eq:AlgLips-2} hold with  $L_A=\gamma N \sqrt{d}$ and $L_C=L$.
   \end{prop}

 Borrowing  terminology from Communication and Information Theory, the communication channel from the master node to the workers is called a broadcast channel, whereas the reverse is called a multiple-access channel (MAC)~\cite{el2011network}. Generally speaking, the main bottleneck in this two-way communication system (between workers and master node) is the MAC channel where there is a risk of co-channel interference if multiple workers simultaneously send their information messages. In the case of strong interference, also called collision, the information messages might not be decodable at the master node and the workers should re-transmit their messages, consuming  much energy if they are battery-powered. These challenges are exacerbated in wireless networks due to random channel attenuation and background noise~\cite{akyildiz2002wireless}. Consequently, state-of-the-art communication technologies need additional coordination and  advanced coding techniques in the MAC channel.
 In the algorithm in Eq.~\eqref{eq:DL_star},  we  focus on the communication from the worker nodes to the master node (the MAC channel) and do not specially formulate how the master node broadcasts $\vec{x}^k$ and assume zero latency (a conventional ideal broadcast channel). 
   However, our framework can easily be adjusted to include the latency due to the broadcast channel of the master.

\subsubsection{Distributed Learning over a Network: Dual Decomposition}  \label{sec:ExamplesDL-DD}   
Distributed algorithms where the nodes cooperatively solve  the problem in Eq.~\eqref{eq:mainOptProb} over a connected communication network  $(\mathcal{N},\mathcal{E})$ can also be modelled in the form of Eq.~\eqref{eq:ItAlg1}, e.g., using dual decomposition~\cite{uribe2018dual}, ADMM~\cite{shi2014linear}, or  distributed consensus  gradient methods~\cite{qu2018harnessing}.
 All of these algorithms have linear convergence rate under Assumption~\ref{assumption:smooth}. 
  We  illustrate this for the distributed dual decomposition. Suppose that the nodes communicate over an undirected  graph and
   let $\vec{W}\in \R^{N\times N}$ denote the Laplacian matrix of that graph, i.e., $\vec{W}_{ij}=-1$ if $(i,j)\in \mathcal{E}$,
   $\vec{W}_{ij}=|\mathcal{N}_i|$ if $i=j$, and $\vec{W}_{ij}=0$ otherwise.
  We can then write the problem in Eq.~\eqref{eq:mainOptProb} equivalently   as
\begin{equation} \label{eq:mainOptProb-dual-IP}
\begin{aligned}
& \underset{c}{\text{minimize}}
& &   \sum_{i=1}^N f_i(\vec{c}_i)  \\
&  \text{subject to} && \bec{W}\vec{c}=\vec{0} 
\end{aligned}
\end{equation}
where $ \bec{W}=\vec{W}\otimes \vec{I} $, $\vec{c}=(\vec{c}_1,\ldots,\vec{c}_N)$, and $\vec{c}_i$ is a local copy node $i$ has of the variable $\vec{z}$.
 The constraint $\bec{W}\vec{c}=\vec{0}$ ensures the consensus between all the local copies $\vec{c}_i$ for $i=1,\ldots,N$, provided that the network is connected.
 We obtain a distributed algorithm by considering the dual of the problem in Eq.~\eqref{eq:mainOptProb-dual-IP}, we illustrate the details of algorithm derivation in Appendix. 
   The algorithm reduces to the  following steps. 
  Initialize $\vec{x}_i^0=\vec{0}$ and $\vec{c}_i^{0} = \text{argmin}_{\vec{c}_i} ~f_i(\vec{c}_i)$  and for $k\in \N$
\begin{equation}
\begin{aligned} \label{eq:ItAlg2Grad}
    \vec{x}_i^{k+1} =&A_i(\vec{c}_{\mathcal{N}_i}^{k},    \vec{x}_i^k):=  \vec{x}_i^k+\gamma \sum_{j=1}^n W_{ij}  \vec{c}_i^{k}  \\
   \vec{c}_i^{k+1} =& C_i(\vec{x}_i^{k+1}):=\underset{\vec{c}_i}{\text{argmin}}~ f_i(\vec{c}_i)+ \langle \vec{c}_i, \vec{x}_i^{k+1} \rangle.
\end{aligned}
\end{equation}
  Note that here $\vec{x}_i$ are the dual variables and $\vec{c}_i$ are primal problem variables.  The algorithm is in the form of Eq.~\eqref{eq:ItAlg2} which is a special case of our algorithm framework in Eq.~\eqref{eq:ItAlg1}.
 It is also linear convergent as presented in the following theorem, proved in Appendix~\ref{Appendix:Proof_Prop2}. 
\begin{prop} \label{THM:DD-alg}
    If we set $\mathcal{X}:=\texttt{Im}(\vec{W}\otimes \vec{I})$ then
    $A(\vec{c},\vec{x})\in \mathcal{X}$ for any $\vec{c}\in \R^{d}$ and $\vec{x}\in \mathcal{X}$.
    Moreover, if we choose 
    $$\gamma= \frac{2L\mu}{\mu\lambda_{\min}^+(\vec{W})+L\lambda_{\max}(\vec{W})}$$ then for all $\vec{x}\in  \mathcal{X}$ we have 
   \begin{align}
       ||A(C(\vec{x}),\vec{x})-\vec{x}^{\star}||_{\vec{M}} \leq&~~ \sigma ||\vec{x}-\vec{x}^{\star}||_{\vec{M}},
   \end{align}
 where
 $\sigma=1-2/(\kappa({\vec{W}})+1)$,  $\kappa({\vec{W}})=\lambda_{\max}(\vec{W}) L/ (\mu \lambda_{\min}^+(\vec{W}))$,
  $||\cdot||_{\vec{M}}$ is a norm on $\mathcal{X}$, 
and   $\vec{x}^{\star}$ is the unique fixed point of $A(C(\vec{x}),\vec{x})$ in $\mathcal{X}$.
   There exist constants $L_A$ and $L_C$ such that Eq.~\eqref{eq:AlgLips-1} and Eq.~\eqref{eq:AlgLips-2} hold true. 
   See Appendix~\ref{Appendix:Proof_Prop2} for $L_A$, $L_C$, and the definition of
 $||\cdot||_{\vec{M}}$.
 \end{prop}
\noindent The proposition shows that the distributed algorithm in Eq.~\eqref{eq:ItAlg2Grad} over the communication graph $(\mathcal{N},\mathcal{E})$ is indeed contractive.  Note that the contractivity parameter $\sigma$ depends on the spectral properties of the graph via $\kappa({\vec{W}})$.  For example, for a complete graphs, we have $\lambda_{\max}(\vec{W}) = \lambda_{\min}^+(\vec{W})=N$ meaning that $\kappa({\vec{W}})=L/\mu$, which is the condition number of the convex problem in Eq.~\eqref{eq:mainOptProb-dual-IP}.  In other words, for complete graphs the algorithm in Eq.~\eqref{eq:ItAlg2Grad} has the same convergence rate as the decentralized algorithm in Eq.~\eqref{eq:DL_star}. For general graphs, the convergence rate becomes slower as $\lambda_{\max}(\vec{W})/\lambda_{\min}^+(\vec{W})$ grows. 
Note that since the graph is connected, $\lambda_{\min}^+(\vec{W})$ is the second smallest eigenvalue of the graph Laplacian matrix (also known as the algebraic connectivity of the graph), which is often used to characterize  connectivity or propagation times in networks~\cite[Appendix B]{sayed2014diffusion}.

\section{Maintaining Convergence Under Limited Communication} \label{Sec:Maintain}

In this section, we illustrate how we can limit the communication of algorithms with $\sigma$-linear convergence rate (Definition~\ref{Defin:LinearConvergence}) to a few bits per iteration while still maintaining the linear convergence rate. 
 We first illustrate how we quantize the communication in the algorithm (Section~\ref{sec:LimComALg}) and then provide our main convergence results (Section~\ref{sec:MainRes}). 

\subsection{Limited Communication Algorithms} \label{sec:LimComALg}

\begin{figure}


\centering
\begin{tikzpicture}
\draw[step=0.5cm,gray!40,thin] (-2,-2) grid (2,2);
\draw[thin] (-2,-2) -- (2,-2) -- (2,2) -- (-2,2) -- (-2,-2);
\draw[thick,<->,>=stealth] (-2,0) -- (0,0); 
\draw[thick,<->,>=stealth] (0,-2) -- (0,0);
\node[] at (0,0) {$\bullet$};
\node[] at (0.3,0.3) {$\vec{q}_i^{k}$};
\node[] at (-1.0,1.5) {\color{blue}$\bullet$};
\node[] at (-1.2,1.65) {\color{blue}$\vec{q}_i^{k+1}$};
\node[] at (-0.8,1.35) {\color{red}$\bullet$};
\node[] at (-0.3,1.39) {\color{red}$\vec{c}_i^{k}$};
\node[] at (0.5,-1) {$r$};
\node[] at (-1,0.3) {$r$};

\draw[thick,<->,>=stealth] (1,1.5) -- (1.5,1.5);
\node[] at (1.25,1.8) {$\delta^{k}$};

\node[] at (-2.2,-2.2) {$1$};
\node[] at (-2.2,-1.5) {$2$};
\node[] at (-2.2,-1) {$3$};
\node[] at (-2.2,-0.5) {$4$};
\node[] at (-2.2,0) {$5$};
\node[] at (-2.2,1) {$\vdots$};
\node[] at (-2.3,2) {$2^b$};

\node[] at (-1.5,-2.2) {$2$};
\node[] at (-1,-2.2) {$3$};
\node[] at (-0.5,-2.2) {$4$};
\node[] at (0,-2.2) {$5$};
\node[] at (1,-2.3) {$\cdots$};
\node[] at (2.05,-2.25) {$2^b$};

\end{tikzpicture}
  \caption{The $b$-bit (per entry) quantization $\vec{q}_i^{k+1}=\texttt{quant}_i(\vec{c}_i^k,\vec{q}^k_i,r^k,b^k)$ projects the point $\vec{c}_i^k$ to the closest point on the grid.
  The grid is centered at $\vec{q}_i^k$ (which is available to the receiver from the previous iteration) and has the width $2r^k$ and the length between points is  $\delta=2r^k/( 2^b-1)$. 
  } 
 \label{fig:Quant}
\end{figure}
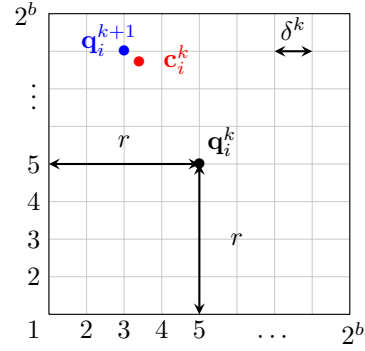

We consider the following $b$-bit (per dimension) quantized version of the algorithms in the form of Eq.~\eqref{eq:ItAlg1}. 
We initialize by setting $\vec{x}^0\in \mathcal{X}$ and  $\vec{q}_i^0=\vec{c}_i^0= C_i(\vec{x}_i^0)$ for $i=1,\ldots,N$ and then do the following iterations for $k\in \N$:
\begin{subequations} \label{eq:ItAlg3}
 \begin{align}
     \vec{x}^{k+1}=& A_i(\vec{q}^{k},    \vec{x}^k), \label{subeq:ItAlg3a} \\
    \vec{c}_i^{k+1}=& C_i(\vec{x}^{k+1}) \label{subeq:ItAlg3b} \\
    \vec{q}_i^{k+1}=& \texttt{quant}_i(\vec{c}_i^{k+1},\vec{q}_i^k,r^k,b) \label{subeq:ItAlg3c}
 \end{align}
\end{subequations}
where $\vec{q}=(\vec{q}_1,\ldots,\vec{q}_N)$.
 The steps in Eq.~\eqref{subeq:ItAlg3a} and Eq.~\eqref{subeq:ItAlg3b}  are essentially the same as  the original unquantized algorithm in Eq.~\eqref{eq:ItAlg1}.
 The main difference is in the quantization step in  Eq.~\eqref{subeq:ItAlg3c}.
 The  variable  $\vec{q}_i^{k+1}$ denotes a $b$-bit (per dimension) quantization   of $ \vec{c}_i^{k+1}$.
 The quantization is done by  using the quantization function $\texttt{quant}_i(\vec{c}_i,\vec{q}_i,r,b)$ that
 projects the point $\vec{c}_i$ to the closest point on the grid illustrated in Figure~\ref{fig:Quant}, which can be represented by $db$-bits.
  The grid is centered at $\vec{q}_i^{k}$, which is available to the receiver from the previous iteration.
  The grid  has the width $2r^k$ and therefore  $r^k$ controls the accuracy of the quantization. 
 We will show  how we can control the accuracy $r^k$ so that its decrease will follow the convergence  of the algorithm. 
 This  allows us to maintain the convergence rate of the algorithms even though we use only $b$ bits to quantize and  communicate  per iteration. 
We can formally define the quantization as follows:
 \begin{defin}
     Let
     $\texttt{quant}_i:\R^d\times\R^d\times \R_+\times \N\rightarrow\R^d$ be the quantization function defined component-wise as follows, where $ \delta(r,b):=r/( 2^b -1)$,
     \begin{align*} &[\texttt{quant}_i(\vec{c},\vec{q},r,b)]_j      = \\
                           &          \begin{cases} \vec{q}_j-r  & \text{if } \vec{q}_j\leq \vec{c}_j-r +\delta(r,b) \\
                                     \vec{q}_j+r  & \text{if } \vec{q}_j \geq \vec{c}_j+r- \delta(r,b) \\
                                     \vec{q}_j- r + 2\delta(r,b)\left\lfloor  \frac{\vec{c}_j-\vec{q}_j+r+\delta(r,b)}{2\delta(r,b)}  \right\rfloor
                                                     & \text{otherwise.}               \end{cases}
     \end{align*}
 \end{defin}
 The following result connects the number of quantization bits $b$ and the precision  of the quantization.
 \begin{lemma} \label{lemma:quant}
  Let $\vec{q}_i\in \R^d$ be given for some $i=1 ,\ldots,N$. 
  Then for all $\vec{c}_i\in \R^d$,  
   such that $||\vec{c}_i-\vec{q}_i||_{\infty} \leq r$ 
   we have
  \begin{align*}
     || \texttt{quant}_i(\vec{c}_i,\vec{q}_i,r,b) -\vec{c}_i||_{\infty} \leq& \frac{ r}{2^b-1}. 
 \end{align*}
\end{lemma}

\subsection{Main Result: Maintaining the Linear Convergence} \label{sec:MainRes}

 We now illustrate our main results.
 We first provide few assumptions on what information must be available a priori to running the algorithm so that the quantization can be performed.
 \begin{assumption} \label{assumption:contr}
 The algorithm is $A(C(\vec{x}),\vec{x})$ is  $\sigma$-linear (Definition~\ref{Defin:LinearConvergence}) and the following
 information is available before running the algorithm:
     a) The parameter $\sigma$; 
      b) The initialization of the quantization variable  $\vec{q}_i^0$ ~for $i=1,\ldots,N$; 
       c) A bound $D$ such that either i) $||\vec{x}^1-\vec{x}^{0}||\leq (1-\sigma)D$ or ii) $||\vec{x}^0-\vec{x}^{\star}||\leq D$. 
 \end{assumption}

 That is, we assume that  the parameters $\sigma$, $D$,  and the initialization $\vec{q}_i^0$, for $i=1,\ldots,N$, are known before running the algorithm.
 The convergence rate $\sigma$ can often be computed easily in advance, e.g., similarly as we did in Theorem~\ref{Theorem:Grad1} and~\ref{THM:DD-alg} in Section~\ref{section:Application}. 
 Prior knowledge of  $\vec{q}_i^0$ is also reasonable since setting its value is part of  initializing the algorithm.  
 To obtain the prior knowledge  of $D$, we note that it suffices to find any upper bound on either $||\vec{x}^1-\vec{x}^0||$ or $||\vec{x}^0-\vec{x}^{\star}||$.
 In optimization we can often use the problem structure to bound $||\vec{x}^0-\vec{x}^{\star}||$.
 For example, the optimization problem in Eq.~\eqref{eq:mainOptProb} is
  $N\mu$-strongly convex so from~\cite[Theorem~2.1.8]{nesterov2013introductory} we have
\begin{equation}\label{equation:Wbound}
  ||\vec{x}^0-\vec{x}^{\star}||^2 \leq  (2/\mu)   (F(\vec{x}^0)-F(\vec{x}^{\star}) )\leq(2/ \mu)  F(\vec{x}^0)=:D,
\end{equation}
  where the final inequality can be obtained if   $F(\vec{x})\geq 0$ for all $\vec{x}$, which is usually the case in ML.
%
 The function value $F(\vec{x}^0)$ is often easily available, e.g., in logistic regression $F(\vec{x}^0)= \log(2)$ if $\vec{x}^0=\vec{0}$, see Section~\ref{Sec:NumericalExample}.
  Similarly, if the algorithm projects the iterates to a compact set, which is often done in constrained optimization, then the diameter of that set gives us the bound $D$.
 Moreover,  bounding  $||\vec{x}^1-\vec{x}^0||$ only requires us, at worst, to do some initial coordinations.  
 We note that  c-i) implies c-ii) since no quantization is done at the first iteration so we have
 $$||\vec{x}^0-\vec{x}^{\star}|| \leq ||\vec{x}^0 - \vec{x}^1|| +  ||\vec{x}^{1}-\vec{x}^{\star}||  \leq ||\vec{x}^0 - \vec{x}^1|| + \sigma ||\vec{x}^{0}-\vec{x}^{\star}||$$
 implying that $||\vec{x}^0-\vec{x}^{\star}|| \leq 1/(1-\sigma) ||\vec{x}^0 - \vec{x}^1||$.

 We obtain the
  following linear convergence rate result for the quantized algorithm. 
   \begin{theorem}\label{MainTheorem}
      Consider the quantized algorithm in Eq.~\eqref{eq:ItAlg3}
      and suppose that Assumption~\ref{assumption:contr} holds.
    Set
       \begin{align} \label{eq:def_r}
  r^k =  \frac{K}{L_A} \alpha(b)^{k+1}  D ~~~\text{ where } ~~~ \alpha(b)=\frac{K}{2^{b}-1}+\sigma \\
    \text{ and }~~K=\max \left\{1,\frac{2L_A L_C}{\sigma}\right\}.~~~~~~~~~ \label{eq:def_K}
\end{align}
 Then the following holds:\\
      \begin{align}
        ||\vec{x}^{k}-\vec{x}^{\star}|| \leq& \alpha(b)^k D, ~~~~~\text{ for all}~~ k\in \N. \label{eq:MainBound1} 
      \end{align}
For any $\epsilon>0$ we have
      \begin{align} \label{eq:epsacc}
         ||\vec{x}^k-\vec{x}^{\star}||\leq& \epsilon~~\text{ for all }~~ k\geq k_{\epsilon}(b):= \frac{1}{1-\alpha(b)} \log\left( \frac{D}{\epsilon}\right).
     \end{align}
%
 \end{theorem}
\begin{proof}
 A proof is presented in Appendix~\ref{App:Proof_of_main1}.
\end{proof}

Theorem~\ref{MainTheorem} states that we can maintain a linear convergence of the algorithms in Definition~\ref{Defin:LinearConvergence} with communicating only a fixed number of bits, $bd$, per iteration. 
 This is possible by adaptively squeezing the grid size based on the local geometry of the optimization landscape, using the designed contraction factor $\sigma$, as well as $L_A$ and $L_C$; see Figure~\ref{fig:Quant}.
 Parameter $\alpha(b)$ plays an important role in the performance of our distributed optimization with quantized information exchange. It is the decrease rate of the grid size (see Eq.~\eqref{eq:def_r}), the convergence rate of the limited communication algorithm (see Eq.~\eqref{eq:MainBound1}).
 In both cases, we need $\alpha(b)<1$ to ensure the convergence of the limited communication algorithm. It is easy to show that there is a critical $b=b_c$ for which $\alpha(b_c)<1$. For any $b> b_c$, $\alpha(b)$ converges exponentially to $\sigma$, the convergence rate of the original unquantized algorithm. Part b) of the theorem  provides a bound on the total number of bits needed to find a given solution accuracy, for a given number of bits per iteration $b$. In particular, we need to run up to $k_{\epsilon}(b)$ iterations to obtain $\epsilon$-accurate solution for any arbitrary accuracy $\epsilon >0$.

Besides the number of iterations, we can measure the total number of bits needed to ensure any $\epsilon>0$ accuracy.
\begin{corollary}
  Define $B_{\epsilon}(b)$ as the total number of communicated  bits (per dimension) needed  to ensure an $\epsilon$-solution using the bound in  Eq.~\eqref{eq:epsacc}, provided that we communicate $b$ bits per iteration. We have
\begin{equation}B_{\epsilon}(b)=b k_{\epsilon}(b). \label{eq:BigB_eps} \end{equation}
\end{corollary}

Note that for given desired solution accuracy $\epsilon>0$, we can easily minimize $B_{\epsilon}(b)$ since it has only one variable $b$.  That is, we can find optimal quantization that ensures convergence using the fewest bits. We explore this further in the numerical experiments in Sections~\ref{section:InterPlay} and~\ref{Sec:NumericalExample}. 

\subsection{
 Decentralized Optimization and Comparison to~\cite{tsitsiklis1987communication}} \label{sec:Comparison}


  We now illustrate how Theorem~\ref{MainTheorem} can be used to provide a dimension dependent upper bound on the
  decentralized optimization algorithms in Section~\ref{sec:ExampleDL}.
  This result generalizes  the upper bound in~\cite{tsitsiklis1987communication}.
  The work in~\cite{tsitsiklis1987communication} considers the constrained optimization problem
\begin{equation} \label{eq:mainOptProb_twonodes}
\begin{aligned}
& \underset{\vec{x}\in \R^d}{\text{minimize}}
& &   f_1(\vec{x})+f_2(\vec{x}) , \\
& \text{subject to} && \vec{x}\in [0,1]^d,
\end{aligned}
\end{equation}
in a network of the two nodes $i=1,2$. The objective functions $\{f_i\}_i$, maintained by node $i$, are strongly-convex and smooth.
The paper studied the following question: \emph{how many bits do the two nodes need to communicate  to reach an $\epsilon$ accurate solution to the optimization problem in Eq.~\eqref{eq:mainOptProb_twonodes}?} 
  To answer this question, the authors of~\cite{tsitsiklis1987communication} first provide a lower bound showing that to find an $\epsilon$-accurate solution, all algorithms need to communicate at least
 \begin{equation} \label{eq:TL_lower_bound_}
      \Omega \left( d \left(  \log(d)+\log\left(\frac{1}{\epsilon} \right)\right) \right)\frac{\text{bits}}{\text{node}}.
 \end{equation}
  The authors also show that an $\epsilon$-accurate solution can be achieved with an algorithm that communicates\footnote{Recall that $\kappa=L/\mu$ is the condition number of the optimization problem. In~\cite{tsitsiklis1987communication} $\kappa$ does not show up in the upper bound. However, a thorough inspection of the proof reveals this dependence on $\kappa$.  }
  \begin{equation}
  \mathcal{O}\left( \log(\kappa d) \kappa d \left(\log(d)+\log\left(\frac{1}{\epsilon} \right) \right) \right) \frac{\text{bits}}{\text{node}}. \label{eq:Oddep}
\end{equation}
  The upper bound is tight, except for factor $\kappa \log(d \kappa)$.
  The algorithm that is used to achieve this upper bound is  a projected  gradient method 
  with a quantization similar to Eq.~\eqref{eq:ItAlg3}.


 Our results can generalize this upper bound to a multiple nodes.
In the following discussion, we assume that each node can communicate to every other node.
 We first consider the unconstrained case of Section~\ref{sec:ExampleDL}. 
 \begin{corollary} \label{Corr:TL1}
  Consider the quantized algorithm in Eq.~\eqref{eq:ItAlg3} with  $A(\cdot)$ and $C_i(\cdot)$ from Eq.~\eqref{eq:DL_star}, $\gamma=2/(N(\mu+L))$, $K$ from Eq.~\eqref{eq:def_K}, $r^k$ from Eq.~\eqref{eq:def_r},  $b= \lceil \log_2(24 (\kappa+1) \sqrt{d}) \rceil$, and $\kappa=L/\mu\geq 2$. 
 For any $\epsilon>0$,  we find $x_{\epsilon}$ such that $||x_{\epsilon}-x^{\star}||\leq \epsilon$ after the algorithm has communicated\footnote{The condition number $\kappa$ my be replaced by any upper bound on $\kappa$. Hence, the assumption $\kappa=L/\mu\geq 2$ does not restrict the results, in this case we may replace $\kappa$ with its upper bound $2$. } 
   $$\mathcal{O}\left( \log \left(\kappa d \right) \kappa d \left(\log \left(D\right)+\log \left(\frac{1}{\epsilon}\right)\right)  \right)\frac{\text{ bits}}{\text{nodes}}.$$
 \end{corollary}
\begin{proof}
   A proof is presented in Appendix~\ref{App:CorrProofs}.
\end{proof}
This corollary shows that the quantized version of the unconstrained decentralized gradient method archives a similar upper bound as the projected gradient method in~\cite{tsitsiklis1987communication} (see Eq.~\eqref{eq:Oddep}), even for $N>2$.
  The only difference  is that the $\log(d)$ factor in Eq.~\eqref{eq:Oddep} has been replaced with $\log(D)$.
  The bound $D$ does not always depend on the dimension $d$, e.g., the upper bound in Eq.~\eqref{equation:Wbound}.  

We can generalize the two-node optimization algorithm~\cite{tsitsiklis1987communication} to a network as
\begin{equation} \label{eq:mainOptProb_multinodes}
\begin{aligned}
& \underset{\vec{x}\in \R^d}{\text{minimize}}
& &  \sum_{i=1}^n  f_i(\vec{x}) \:, \\
& \text{subject to} && \vec{x}\in [0,1]^d.
\end{aligned}
\end{equation}
The optimization algorithm we use is projected gradient method, which we write in our  algorithm framework by  setting
   \begin{equation}
   \label{eq:DL_starTL}
   \begin{aligned}
    A(\vec{c},\vec{x})= \left\lceil \vec{x}- \gamma\sum_{i=1}^N \vec{c}_i \right\rceil_{[0,1]^d} \text{ and }C_i(\vec{x})=  \nabla f_i(\vec{x})
  \end{aligned}
  \end{equation}
  for  $i=1,\ldots,N$.
  Note that for any $\vec{x}^0\in [0,1]^d$ we satisfy Assumption~\ref{assumption:contr} with $D=\sqrt{d}$, since $||\vec{z}-\vec{y}||_2\leq\sqrt{d}$ for any $\vec{z},\vec{y}\in[0,1]^d$.
  The quantized version of this algorithm leads to the following result.
 \begin{corollary}  \label{Corr:TL2}
 Consider the quantized algorithm in Eq.~\eqref{eq:ItAlg3} with  $A(\cdot)$ and $C_i(\cdot)$ from Eq.~\eqref{eq:DL_star}, $K$ from Eq.~\eqref{eq:def_K}, $r^k$ from Eq.~\eqref{eq:def_r}, $\gamma=1/(NL)$, $b= \lceil \log_2(24 (\kappa+1) \sqrt{d}) \rceil$, and $\kappa=L/\mu\geq 2$.
 For any $\epsilon>0$,  we find $\vec{x}_{\epsilon}$ such that $\|\vec{x}_{\epsilon}-\vec{x}^{\star}\|\leq \epsilon$ after the algorithm has communicated
 \begin{equation*}
 \mathcal{O}\left( \log(\kappa d) \kappa d \left(\log(d)+\log\left(\frac{1}{\epsilon} \right) \right) \right) \frac{\text{bits}}{\text{node}}.
 \end{equation*}
 \end{corollary}
\begin{proof} A proof is presented in Appendix~\ref{App:CorrProofs}.
\end{proof}
This corollary illustrates that we can use  the results developed in this paper to extend the upper bound in~\cite{tsitsiklis1987communication} to multi-nodes. 
 The upper bound almost  matches the lower in Eq.~\eqref{eq:TL_lower_bound_}, within a $\log(\kappa d) \kappa$ factor.
 %
 %
  The linear  dependence on the condition number $\kappa$ is an artifact of the gradient method and cannot be improved unless we use another optimization algorithm, e.g., accelerated gradient method or a conjugate gradient method.
  Improving  the  $\log(\kappa d)$ factor might be possible. 

\section{Transmission-time Complexity} \label{Sec:TTC}
Although quantization (lossy compression) saves communication resources, it may fail in characterizing
the actual cost of running distributed  algorithms.
 This is because, from the engineering perspective, the tangible costs are usually the latency (time needed to run the algorithm) or energy consumption. 
  In the following, we focus on the latency and characterize the transmission-time convergence of running the distributed algorithms from the previous section. We introduce the  transmission time model that we use in subsection~\ref{section:Time-Model}. We provide  convergence rate results in terms of transmission time with and without communication errors in subsections~\ref{section:Time-Error-Free} and~\ref{section:Time-Faulty}, respectively.

\subsection{Transmission-Time Model} \label{section:Time-Model}

 To perform an iteration of the quantized algorithms in Eq.~\eqref{eq:ItAlg3} each node needs to communicate a package of $$n:=bd+\theta \text{ bits.}$$
  In particular,  $db$ bits are used to transmit  the quantized message $\vec{q}_i^k$, $b$ per dimension.
  The remaining $\theta$ bits model all protocol overheads, including packet headers (source and destination IDs), parity bits for channel coding, and scheduling overheads~\cite{Jiang2019LowLatency}. These overheads, among other purposes, are added to realize reliable communications over an unreliable communication link and to regulate the data transmissions of multiple nodes.
  We discuss the overhead in more detail in Section~\ref{section:InterPlay}.

   To characterize the transmission time of running the algorithms we need to define the transmission rate.
   To simplify the discussion we let the transmission rate be the same for each communication link.\footnote{This is without loss of generality, since we only need to consider the transmission rate of slowest communication link. } 
 The \emph{transmission rate} is denoted by the function
     $$R(n,p) =\frac{\text{communicated bits}}{\text{second}},$$
  where $p\in(0,1)$ is the probability that each $n$-bit communication package fails in transmission.
  We discuss specific forms of the rate function $R(n,p)$ in Section~\ref{section:InterPlay}.
To quantify the transmission time per iteration,
we  introduce the \emph{delay function} $\Delta (n,p)$.
In particular,  the time needed to perform 1 iteration of the algorithm  is denoted by 
\begin{equation}\label{EQ:Delta1}
\Delta (n,p) := \frac{n}{R(n,p)} ~~\text{seconds}
\end{equation}
and  performing $k$ iterations takes
\begin{equation} t := k\Delta(n,p)=\frac{kn}{R(n,p)} ~~\text{seconds} .   \label{EQ:Time1} \end{equation}
 If we let  $ \vec{x}_i(t)$ be the decision of node $i$ after $t$ transmission seconds then
%
 \begin{align}
 \vec{x}_i(t)=& \vec{x}_i^k,~~~\text{ for } ~~~t\in \big[k\Delta(n,p),(k+1)\Delta(n,p)\big), \\
 \vec{x}(t)=&(\vec{x}_1(t),\ldots,\vec{x}_N(t)).
\end{align}
We now explore the algorithm convergence in terms of transmission time.

\subsection{Error-free Transmission} \label{section:Time-Error-Free}
We start by considering an error-free communication, i.e., when $p=0$.
In this case, we drop the variable $p$ from the notation $R(\cdot)$ and $\Delta(\cdot)$.
 We get the following  transmission-time convergence.

 \begin{theorem}\label{MainTheorem-time} 
  Consider the quantized algorithm in Eq.~\eqref{eq:ItAlg3}.   Suppose that  Assumption~\ref{assumption:contr} holds and that the parameters $r^k$, $\alpha(b)$, and $K$ are chosen as in Eq.~\eqref{eq:def_r} and Eq.~\eqref{eq:def_K} in Theorem~\ref{MainTheorem}.
Then we have the following convergence in terms of transmission time $t$
      \begin{align} \label{eq:MainBound2-Time}
                  ||\vec{x}(t)-\vec{x}^{\star}|| \leq& \left(  \left[\frac{K}{2^b-1}+ \sigma  \right]^{R(n)/n} \right)^{t-\Delta(n)} D.
      \end{align}
      Moreover, for any $\epsilon>0$ following holds:
      \begin{align} \label{eq:eq:mainBound2-time2}
          ||\vec{x}(t)-\vec{x}^{\star}||\leq& \epsilon~~\text{ for all }~~ t\geq T_{\epsilon}(b,\theta),
     \end{align}
  where 
\begin{align} \label{eq:TimeTeps}
   T_{\epsilon}(b,\theta) :=&~ k_{\epsilon}(b)  \Delta(b+\theta) ~~\text{seconds},
\end{align}
 with $k_{\epsilon}(b)$  defined in Eq.~\eqref{eq:epsacc}.
 \end{theorem}
\begin{proof}
 This result is obtained by using Eq.~\eqref{EQ:Time1} in Theorem~\ref{MainTheorem}.
\end{proof}
 Eq.~\eqref{eq:MainBound2-Time} provides us with an upper bound on the convergence of $\vec{x}(t)$ to fixed point $\vec{x}^{\star}$ in terms of transmission time.
  The algorithm converges linearly in terms of transmission time with  the rate
  \begin{align} \label{eq:Rate}
  \rho(b,\theta) =  \left[\frac{K}{2^b-1}+\sigma  \right]^{R(db+\theta)/(db+\theta)}.
 \end{align}
  Unlike the convergence rate in Eq.~\eqref{eq:MainBound1} of Theorem~\ref{MainTheorem},
  a finer quantization (larger $b$) does not necessarily ensure a faster convergence rate anymore.
     This is because the convergence rate $\rho(b,\theta)$ for the transmission time captures the  tradeoff between spending more time to send large packets with more information (larger $b$) at the cost of doing fewer iterations per second  or spending less time per iteration (by sending a packet with less information) but achieving  more iterations per seconds.
  Even though $\rho(b,\theta)$ seems complicated, it is easily optimized numerically since it is a function of only 2 variables. 

  Eq.~\eqref{eq:TimeTeps}  is an upper bound on the transmission time (in seconds) needed to reach an $\epsilon$-accurate solution.  In comparison to  Theorem~\ref{MainTheorem}, the transmission-time needed to find an $\epsilon$-optimal solution is the iteration count $k_{\epsilon}(b)$ multiplied by the nonlinear delay function $\Delta(n)$.
  We will see in  Section~\ref{section:InterPlay} how the nonlinearities in $\Delta(n)$ affect the transmission rate.

 \subsection{Faulty Transmission} \label{section:Time-Faulty}

 We now consider the transmission time convergence in the presence of  message failures, i.e., $p>0$. 
In particular, each transmitted message from node $i$ to node $j$ fails with probability $p$. We assume that the success of the transmissions are independent events across time and among different transmitter-receiver pairs. 
%

 To perform the algorithm in Eq.~\eqref{eq:ItAlg3}  a successful communication of each message $\vec{q}_i^k$ is required.
 In case of messages failures, the nodes need to re-transmit their messages until success.
 We let $m_k$ denote the number of  re-transmissions of messages at iteration $k$. 
 %
 This means  that the communication time of one  iteration (cf. Eq.~\eqref{EQ:Delta1} and Eq.~\eqref{EQ:Time1}) is $m_i\Delta (n,p)$ seconds
and  performing $k$ iterations takes
\begin{equation} t :=  \Delta(n,p)\sum_{i=0}^k m_i =\frac{n}{R(n,p)} \sum_{i=1}^k m_i  ~~\text{seconds} .   \label{EQ:Time2} \end{equation}
 We consider two chooses of $m_i$:
\begin{enumerate}[1)]
  \item \textbf{Communicate until success:} $m_i$ is the number of communication rounds until every message has bee succesfully received.
  \item \textbf{Constant communication rounds:}  $m_i=m$, where  $m$  is a predefined constant independent of $i$.
\end{enumerate}
 The communication scheme 1) is easily realizable  in  decentralized and parallel computing where a master node handles the coordination, like in the application in Section~\ref{sec:ExampleDL}.
 However, in fully distributed algorithms in networks it is unrealistic that a single node can know when every communication has been successfully received.
 In this case, we can go for alternative 2) and let each node communicate a fixed number of times.
 We now study these two cases separately.


\subsubsection{Communicate until success}
   From our results in Section~\ref{sec:MainRes} and Theorem~\ref{MainTheorem},
   to find an $\epsilon>0$ accurate solution with $b$-bit resolution per dimension we need $k_{\epsilon}(b)$ iterations.
   In terms of transmission time that translates to
  \begin{align} \label{eq:time-US}
     T_\epsilon(b,\theta,p)= \Delta(n,p) \sum_{k=1}^{k_{\epsilon}(b)}  m_k  \text{ seconds.}
  \end{align}
  Here $m_k$ are  realizations of the random variable $M$ indicating the number of communication rounds needed until success at iteration $k$.
  To formally define $M$ we first define the set $\mathcal{L}$ of all communication links (sender-receiver pairs) used in the algorithm.
  Then
   \begin{align} \label{eq:mk_RV}
      M=\max_{l\in \mathcal{L}}~ S_l
   \end{align}
  where $S_l$ is a random variable indicating the number of times the communication in link $l$ is transmitted until it is successfully received.
 It is easily verified that $S_l$ follows a geometric distribution and that
    $P\left[ S_l =m\right]=p^{m-1}(1-p)$ and 
    $P\left[S_l\leq m\right]=1-p^{m}$.
  With this in mind,  we can derive the following bound on the expected transmission-time needed to ensure $\epsilon$-accuracy.
   \begin{theorem}[Communicate until success] \label{Thm:CUS}
   Consider the quantized algorithm in Eq.~\eqref{eq:ItAlg3}.   Suppose that  Assumption~\ref{assumption:contr} holds and that the parameters $r^k$, $\alpha(b)$, and $K$ are chosen as in Eq.~\eqref{eq:def_r} and Eq.~\eqref{eq:def_K} in Theorem~\ref{MainTheorem}.
   Then for any  $\epsilon>0$ we have
      \begin{align}
          ||\vec{x}(t)-\vec{x}^{\star}||\leq& \epsilon~~\text{ for all }~~ t\geq T_{\epsilon}(b,\theta,p),
     \end{align}
where $T_{\epsilon}(b,\theta,p)$ is the random variable defined in Eq.~\eqref{eq:time-US}.
We have the following bounds on $T_{\epsilon}(b,\theta,p)$
   \begin{align*}
          \texttt{LB}_{\epsilon}(b,\theta,p) \leq& E[T_\epsilon(b,\theta,p)]  \leq  \texttt{UB}_{\epsilon}(b,\theta,p)
   \end{align*}
   where
   \begin{align*}
             \texttt{LB}_{\epsilon}(b,\theta,p) =& k_{\epsilon}(b) \times \Delta(b{+}\theta,p)\times  \frac{\log(|\mathcal{L}|)}{\log(1/p)}   \\
            \texttt{UB}_{\epsilon}(b,\theta,p) =&  k_{\epsilon}(b) \times   \Delta(b{+}\theta,p)\times \left( \frac{1{+} \log(|\mathcal{L}|)}{\log(1/p)}{+}1\right) .
   \end{align*}
\end{theorem}
\begin{proof}   A proof is presented in Appendix~\ref{App:ProofTP}.
%
\end{proof}
 The theorem tells us that the expected transmission time needed to obtain an $\epsilon$-accurate solution is bounded by
%
%
%
 \begin{align*}
    E[T_\epsilon(b,\theta,p)]=  k_{\epsilon}(b)     \times   \Delta(b{+}\theta,p)\times \mathcal{O} \left(  \frac{\log(|\mathcal{L}|)}{\log(1/p)}  \right). 
 \end{align*}
 This bounds consists of three multiplicity terms. The first term $k_{\epsilon}(b)$ comes from the iteration count in Theorem~\ref{MainTheorem}, for the number of iterations needed to find an $\epsilon$-solution.
 The second  term is the nonlinear delay function $\Delta(b{+}\theta,p)$ accounting for the achievable communication rate.
 The final term accounts for the re-transmitting the messages until success.
 This means that the transmission time follows nonlinear relationship between $b$, $\theta$, and $p$. This relationship can easily be optimized numerically if we know the delay function $\Delta(\cdot)$, since there are  only three free variables. We explore this in Section~\ref{section:InterPlay}. {The third term also depends logarithmically on the number of communication links $|\mathcal{L}|$. In a network with $N$ nodes, the number of communication links is bounded as  $|\mathcal{L}|\leq N^2$. This means that $\log(|\mathcal{L}|)\leq  \log(N^2)= 2\log(N)$.}


  \subsubsection{Fixed Number of Communication Rounds}

  Consider now communication scheme 2), where each node broadcast  constant  number of times, i.e. $m_k=m$ is a constant.
  In this case, with some probability, a message will be dropped meaning that the quantized algorithm in Section~\ref{Sec:Maintain} is not guaranteed terminate successfully (we discuss how we can relax this assumption after the theorem).
  However, we can make the probability that the algorithm terminates successfully (finds an $\epsilon$-solution) as high as we want by increasing $m$.
  We illustrate this now.
%
 \begin{theorem}[Fixed number of communication rounds] \label{th4}
 Consider the quantized algorithm in Eq.~\eqref{eq:ItAlg3}.   Suppose that  Assumption~\ref{assumption:contr} holds and that the parameters $r^k$, $\alpha(b)$, and $K$ are chosen as in Eq.~\eqref{eq:def_r} and~\eqref{eq:def_K} in Theorem~\ref{MainTheorem}.
   Let $\epsilon>0$ and $\delta\in [0,1)$. Suppose that each node transmits its message
     $$m_{\epsilon}^{\delta}(b,p)=
  \left\lceil  \frac{|\mathcal{L}|k_{\epsilon}(b)}{(1-\delta)\log(1/p)}  \right\rceil~\frac{\text{times}}{\text{iteration}}.$$
  %
   %
   Then with with probability  $\delta$ following holds:
   \begin{align}
       ||\vec{x}(t)-\vec{x}^{\star}||   \leq& \epsilon ~~~~\text{ for all }~~~~t\geq T_{\epsilon}^{\delta}(b,\theta,p)
   \end{align}
   where 
       $T_{\epsilon}^{\delta}(b,\theta,p)=   k_{\epsilon}(b) \times \Delta(b+\theta,p) \times m_{\epsilon}^{\delta}(b,p).$
\end{theorem}
\begin{proof}  A proof is presented in Appendix~\ref{App:ProofTF}.
\end{proof}
 The theorem provides  the transmission time needed to successfully compute an $\epsilon$-solution with $\delta$ probability. The upper bound is similar to the bound in Theorem~\ref{Thm:CUS}, except the final multiplicity factor has been replaced by $m_{\epsilon}^{\delta}$.
 {In particular, as we increase $\delta$ to $1$, the increase in the transmission time is proportional to $\mathcal{O}(1/(1-\delta))$. The transmission time linearly grows with $|\mathcal{L}|$, meaning that  it can grow quadratically with number of nodes in the worst case when   then communication network is a complete graph (every node communicate to every other node).  
  As before, the convergence rate is easily optimized since it has only few parameters.}

 \begin{remark}
  The algorithms considered in this paper are synchronous meaning that we require every communication to successfully terminate at every iteration.  This is only an artifact of the theoretical analysis, in particular of Theorem~\ref{MainTheorem}. However, we may be able to extend Theorem~\ref{MainTheorem} to handle asynchronous communications with delays at the cost of a slower convergence rate.
%
 %
 In particular, if a message is delayed then we center the communicated grid at the last received $\vec{q}_i$ and use the associated grid. But the grid size should be delicately controlled (intuitively decreasing slower) so that the old grid still captures the new communicated information.  This leads to slower convergence which has been traded-off for more robustness to asynchronous communications. We will leave a detailed implementation of these ideas as future work.
 %
 \end{remark}


 \begin{figure*}[t]
    \centering
        \begin{subfigure}[t]{0.3\textwidth}
        \includegraphics[width=\textwidth]{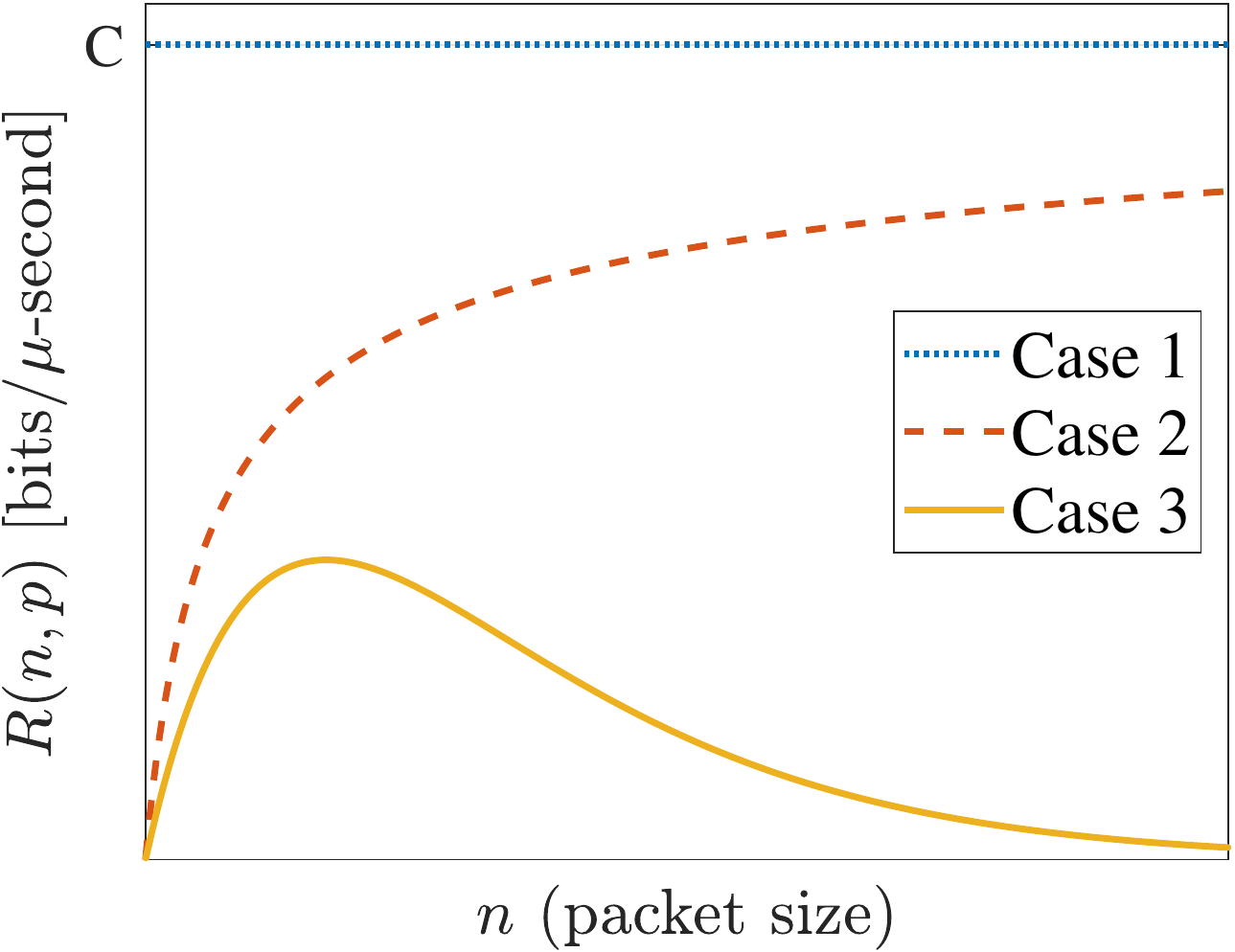}
        \caption{Communication rate vs  package size. }
       \label{fig:rates_cases}
    \end{subfigure}
    \begin{subfigure}[t]{0.3\textwidth}
        \includegraphics[width=\textwidth]{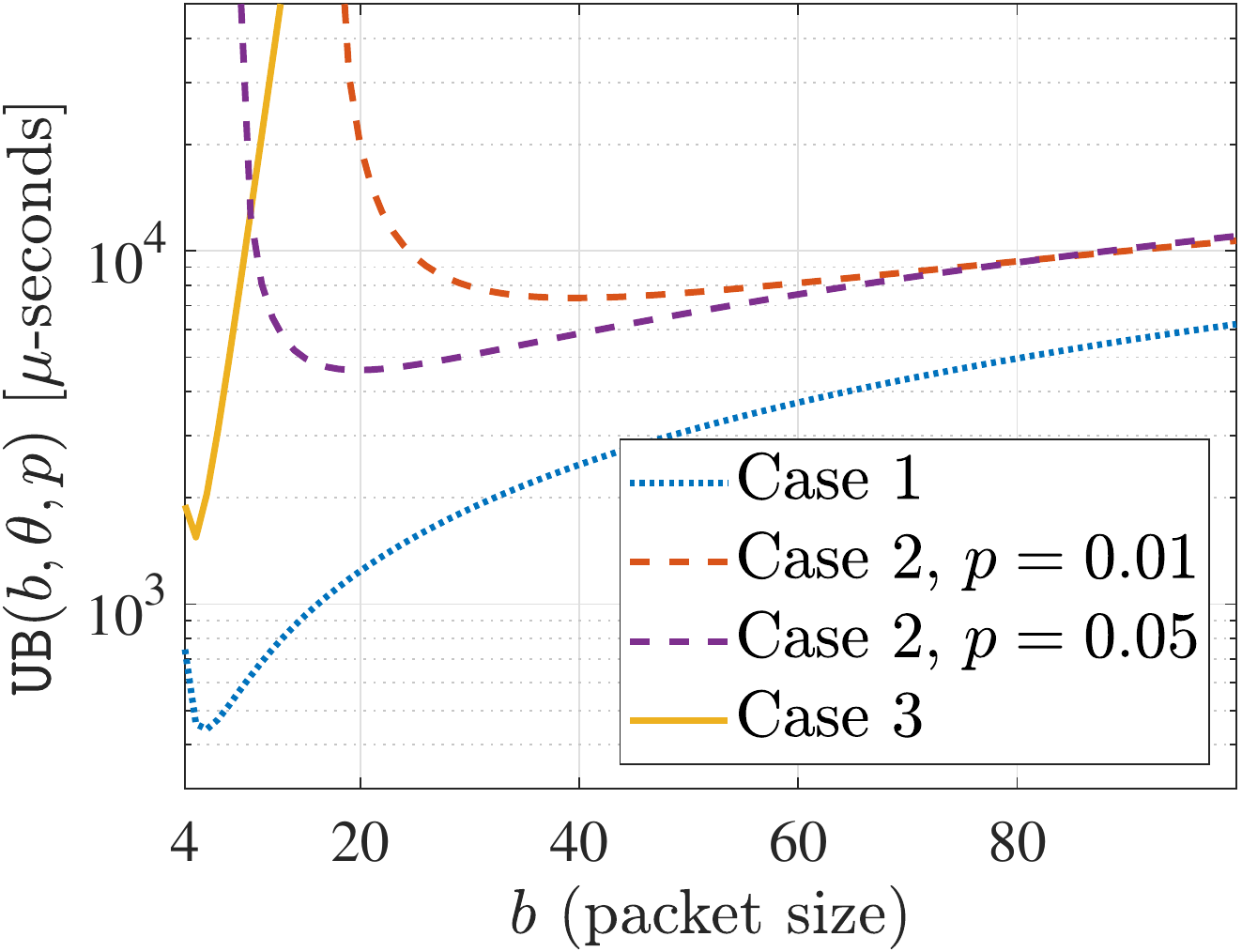}
        \caption{Transmission-time to find $\epsilon$-solution vs packet size (cf. Theorem~\ref{Thm:CUS}). 
        }
        \label{fig:rates_cases_n}
    \end{subfigure}
    \begin{subfigure}[t]{0.3\textwidth}
        \includegraphics[width=\textwidth]{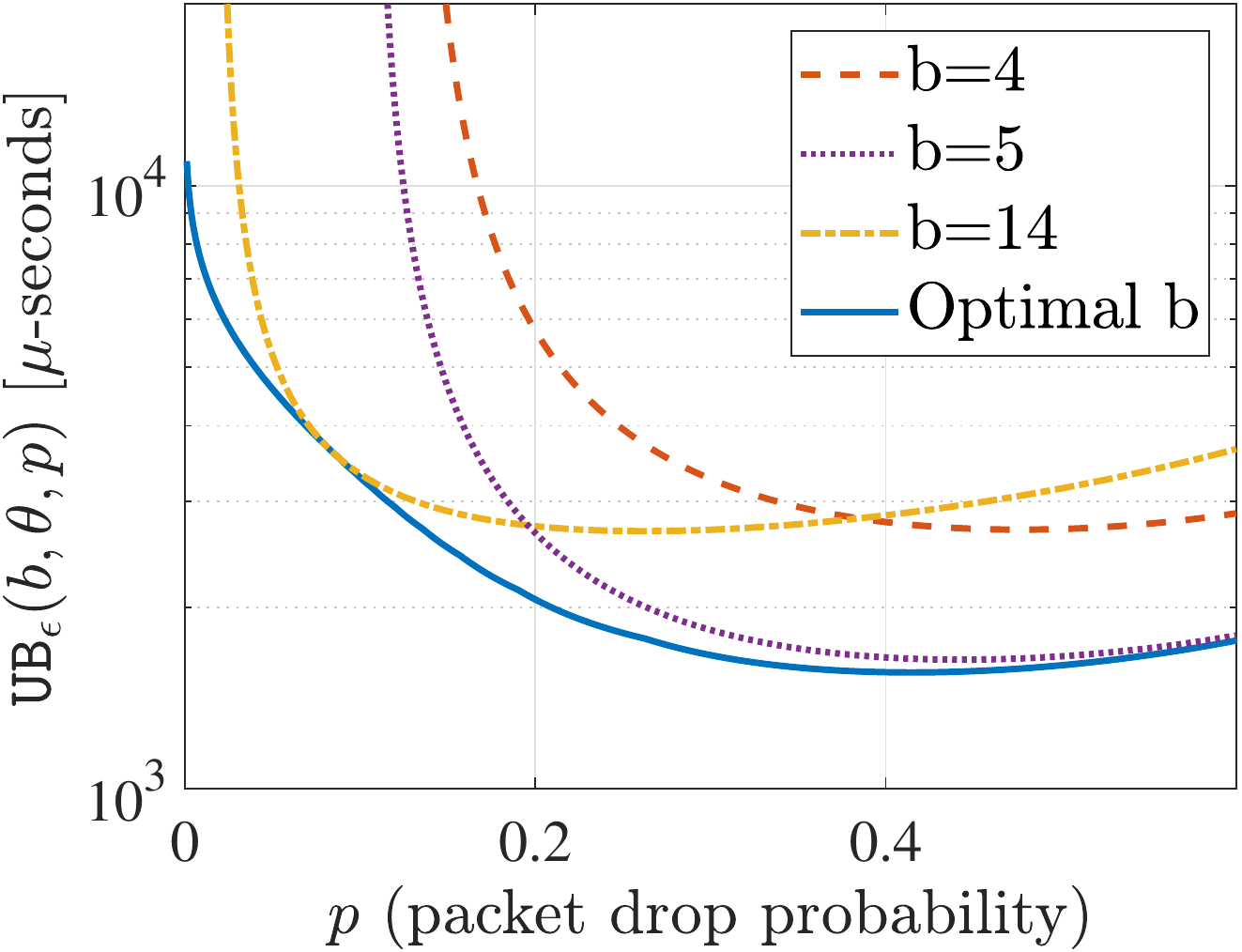}
         \caption{Transmission time to find $\epsilon$-solution vs  packet drop probability  (cf. Theorem~\ref{Thm:CUS}).}
        \label{fig:rates_cases_2}
    \end{subfigure}

    \caption{Illustrations of the transmission-time convergence for different rate functions $R(n,p)$. 
    }
\end{figure*}

\section{Interplay Between Distributed Algorithms \\and Communication Protocol} \label{section:InterPlay} 
In this section, we build on the results of the previous section to explore the interplay between  communication protocols and distributed algorithms.
We first introduce three abstract rate functions, $R(n,p)$, modeling various communication protocols. Afterward, we characterize the relationship between $bd$, $\theta$ $p$ and then provide numerical illustrations.


\subsection{Transmission Rate Models} \label{Sec:Iterplay-R}
Information and communication theory has a rich literature on  rate functions in various communication settings~\cite{el2011network,bertsekas1992data,koubaa2006comprehensive,polyanskiy2010channel}.
The shape of $R(n,p)$ typically depends on 1) the network setting, 2) communication channel, and 3) protocol overheads such as packet headers (source and destination IDs), parity bits for channel coding, and scheduling overheads in the MAC layer.
Most realizations of the transmission rate function  have a shape resembling one of the following three cases.
\begin{itemize}
 \item \textbf{Case 1 (constant):} $R(n,p)=C$ for some constant positive $C$ for all $n$ and $p$.
 \item \textbf{Case 2 (saturating):} $R(n,p)$ for some constant $p>0$ is a non-decreasing function of $n$ that saturates for large $n$. The saturation is slower for smaller $p$.
 \item \textbf{Case 3 (bell-shape):} $R(n,p)$ for some constant $p>0$ increases with $n$ up to some critical $n^{c}$ and gradually decreases to zero afterward.
\end{itemize}
\textbf{Case 1} models the fundamental rate \emph{upper-bound}, which has been the basis for the design of many communication protocols. The seminal works of Shannon showed that for any fixed $p>0$ there exist codes such that $R(n,p)$ converges $C$ as $n$ grows large~\cite{el2011network}. However, these results have two fundamental drawbacks. Firstly, the codes constructed by Shannon cannot be realized in practice. Secondly and more importantly, the result of Shannon requires $n$ to grow to infinity meaning that the latency (the time it takes to deliver a single packet) may grow to infinity. This is clearly not feasible for iterative  algorithms where the updates of each iteration depends on the results of previous iterations. \textbf{Case 2} extends \textbf{Case 1} by looking at finite $n$~\cite{polyanskiy2010channel}. In particular, we can model the achievable rate as~\cite{polyanskiy2010channel}
   \begin{equation}  \label{eq:RCmP}
       R(n,p)=C-P(n,p) ~~\frac{\text{bits}}{\text{second}},
   \end{equation}
 where $P(n,p)$ is a penalty on the transmission-rate that is decreasing with $n$ and increasing with $p$.
   Polyanskiy~\etal~\cite{polyanskiy2010channel} showed that the maximal achievable rate in the finite packet length regime is
   $$R(n,p)=C-\sqrt{\frac{V}{n}} Q^{-1}(p) + \mathcal{O}\left( \frac{\log n }{n} \right) ~ \frac{\text{bits}}{\text{second}},$$
   where $V$ is the so-called channel dispersion, and $Q(\cdot)$ is the Q-function. To give a concrete example, an additive white Gaussian noise (AWGN) channel with average signal-to-noise-ratio of $\kappa=1$~\cite{durisi2016toward} has
  $$C=\log(1+\kappa)=\log(2)~\text{ and }~V=\kappa\frac{2+\kappa}{(1+\kappa)^2} (\log e)^2\approx \frac{3}{2}.$$
 Although this approach substantially extends {\textbf{Case 1}}, this rate function is known for a handful of network settings.
 In practice, various communication standards may have specific limitations on the transmission power and MAC layer protocol and use various approaches to handle incoming interference from unintended transmitters. Consequently, the rate function for most practical communication networks are either unknown or is very different from the upper-bounds provided by {\textbf{Case 1}} and {\textbf{Case 2}}, especially for multiple access channels. \textbf{Case 3} models the rate function of a multiple access channel, regulated by contention-based protocols such as slotted-ALOHA and carrier-sense-multiple-access (CSMA)~\cite{bertsekas1992data,koubaa2006comprehensive}. The rationale behind this model is that increasing the offered network load (e.g., by increasing $n$) improves the rate up a certain point at which the network becomes congested. Further increase of the offered load only increases the collision among packets, leading to a drastic drop of the successful packet reception. The contention-based protocols are heavily used in most modern standards for wireless local area networks, such as IEEE~802.11 family, due to their easy implementation and inherent distributed coordination~\cite{bertsekas1992data}.
Fig.~\ref{fig:rates_cases} illustrates our three example rate models.

%
 %

\subsection{The Relationship between $bd$, $\theta$, and $p$} \label{Sec:Iterplay-bdtp}
In communication systems, there are internal relationships  between the number of payload bits  $bd$, overhead bits $\theta$, and reliability $p$. For example, $bd$ and $\theta$ might be related by an affine function $\theta=A (db)+B$, where $A$ and $B$ are some constants, modeling various overheads and packet headers such as channel coding, source and destination IDs, and scheduling overheads~\cite{Jiang2019LowLatency}. The constant $A$ models the channel coding overhead (to ensure a reliable communication over an unreliable channel) for linear codes. In a more generic class of codes, our affine $\theta$ is a first-order approximation of the overheads required to send $bd$ bits~\cite{roth2006introduction}. Moreover, there exist a non-zero packet failure probability $p$, which is usually a non-linear function of $b$ and $\theta$. The exact expression depends on the transmission scheme and channel model~\cite{roth2006introduction}, as we illustrate below. Parameter $B$ can also model the background traffic, which a node will send in parallel to running the   algorithm.

\subsection{Illustrative Examples} \label{sec:Iterplay-ill}
We now illustrate the impact of different rate functions $R(n,p)$ on the transmission-time convergence.
For illustration purposes,  consider Theorem~\ref{Thm:CUS} with $K=1$, $D=1$, $\epsilon = 0.1$, $\sigma=0.9$, $|\mathcal{L}|=20$, $d=1$, and $\theta=0$,  for the three cases of $R(n,p)$ in Subsection~\ref{Sec:Iterplay-R}.
We consider the following rate functions in  bits per seconds
 \begin{align*}
   R_1(n,p) &=C , ~~~~~~~~~~~~~R_2(n,p)=C-\sqrt{\frac{V}{n}} Q^{-1}(p), \\
   R_3(n,p)&=\frac{n}{5}\exp\left(-\frac{n}{5}\right) ,
 \end{align*}
 and an AWGN channel with signal-to-noise-ratio of 1, leading to $C=\log(2)$ and $V=1.5$ as explained in the last subsection.

 Fig.~\ref{fig:rates_cases} illustrates our three rate models, and Fig.~\ref{fig:rates_cases_n} shows the impact of $b$ on the transmission-time upper bound $\texttt{UB}_{\epsilon}(b,0,p)$.
In \textbf{Case~1}, where the effects of packet size are ignored, it is the best to use a very small $b$, the optimal choice is $b=5$.
In \textbf{Case 2}, where larger packet sizes can achieve higher rates, the optimal $b$ becomes much larger, i.e., $b\approx 40$ and $b\approx 20$, respectively, for $p=0.01$ and $p=0.05$. In \textbf{Case 3}, where sending very large packets by all the nodes may overload the channel (congestion), the optimal value for $b$ reduces. Commered to \textbf{Case 1}, channel congestion and reduced rate of \textbf{Case 3} lead to have a higher penalty after optimal $b$, indicated by a sharper increase in $\texttt{UB}_{\epsilon}(b,0,p)$.
Fig.~\ref{fig:rates_cases_2} demonstrates the effect of $p$ on the transmission-time convergence for rate mode $R_2(n,p)$.
The blue solid curve shows the optimal convergence rate $\texttt{UB}_{\epsilon}(b,0,p)$ for each $p$ by optimizing over $b$.
The results indicate on the importance of optimal $b$ on the convergence time. Interestingly, there is a delicate trade-off between having $p$ too small or too big, which we explore in the next section.


\section{Numerical Experiments}
\label{Sec:NumericalExample}

We test our result by training a logistic regression model on two  data sets: a) \texttt{SP-data}~\cite{anguita2012human} and b) \texttt{MNIST}~\cite{lecun1998gradient}. The \texttt{SP-data} has dimension $d=61$ and consists of $24,075$ points from smart phone sensors used to classify whether the person carrying the phone is moving (walking, running, or   dancing) or not (sitting or standing). We use random $m=10,000$ to train our model and the rest for validation. The \texttt{MNIST} has dimension $d=785$ and consists of pictures of handwritten digits. It has 10 labels, $m=60,000$  samples for training and $10,000$ samples for validation. We have trained a multiclass logistic regression using one-vs-rest approach.

  We consider logistic regression
 with $l_2$  regularization~\cite{koh2007interior}
\begin{equation} \label{eq:mainOptProb3}
\begin{aligned}
& \underset{\vec{z}\in \R^d}{\text{minimize}}
~ F(\vec{z})=  \frac{1}{m}\sum_{i=1}^m \log\left(1+ \exp(-y_i \cdot \vec{z}\tran \vec{v}_i)\right)+  \frac{\rho}{2}|| \vec{z}||_2^2,
\end{aligned}
\end{equation}
where $\rho>0$ is a regularization parameter,  $\vec{v}_i\in \R^d$ the explanatory feature vector, and  $y_i\in \{-1,1\}$ the binary output outcome. Note that $F(\vec{z})$ is $\rho$-strongly convex and has $\rho+V$ Lipschitz continuous gradients, where $V$ can be computed by using that for each data point $i$ the gradient of $\log\left(1+ \exp(-y_i \cdot \vec{z}\tran \vec{v}_i)\right)$ is $V_i$-Lipschitz continuous where $V_i=\|\vec{v}_i\|_2^2/4$. Note that from  Eq.~\eqref{equation:Wbound}, if  $\vec{z}^0=\vec{0}$ then we have $||\vec{z}^0-\vec{z}^{\star}||\leq \sqrt{(2/\rho) \log(2)}=:D.$ The problem has a unique optimizer $\vec{z}^{\star}$, provided that $\rho>0$.  With $\vec{z}^{\star}$ the logistic classifier outputs the label  $y=\texttt{sign}(\vec{z}^{\star\rm{T}} \vec{v} )$  for the new data $\vec{v}$.  We use $\rho=350$ for the \texttt{SP-data} and $\rho=1$ for \texttt{MNIST}, giving $99\%$ and $80\%$ accuracy on the test data, respectively.

\subsection{Communication Complexity}

  We use the quantized  version  of  the decentralized algorithm (Figs.~\ref{fig:Dec} and~\ref{fig:Deca}) and the distributed algorithm (Fig.~\ref{fig:Dist}) in Section~\ref{section:Application}.  In both cases there are $N=20$ nodes and we split the data equally among them.   Unless otherwise mentioned, we choose the parameters according to Theorems 1 and 3, e.g., for the $\texttt{SP-data}$ $\gamma=2/(\mu+L)\approx 6\times 10^{-5}$, $\sigma\approx 0.98$, and $K\approx 15.8$.  We compare the results  to the algorithm without quantizing (but using double precision floating points, 64 bits per dimension).  Fig.~\ref{fig:rates_cases_n} illustrates the convergence of the decentralized algorithm on the \texttt{SP-data}, indicating that our approach can maintain the convergence of unquantized algorithm ($b=64$) by just $b=14$ or $16$.    This means that the quantized version of the algorithm reduces the communicated bits by $75\%$ without slowing down the convergence speed.      Fig.~\ref{fig:Decd} depicts the convergence on the \texttt{MNIST} data, whose sample dimension is 13 times larger than that of the SP-data.   Again, when we communicate $b=14$ or $b=16$ bits per dimension we get almost the same convergence as the unquantized algorithm, suggesting that the results of the paper scale well with  the dimension.

\begin{figure*}
\centering
\begin{minipage}{.74\textwidth}
  \centering
    \begin{subfigure}[t]{0.32\textwidth}
        \includegraphics[width=\textwidth]{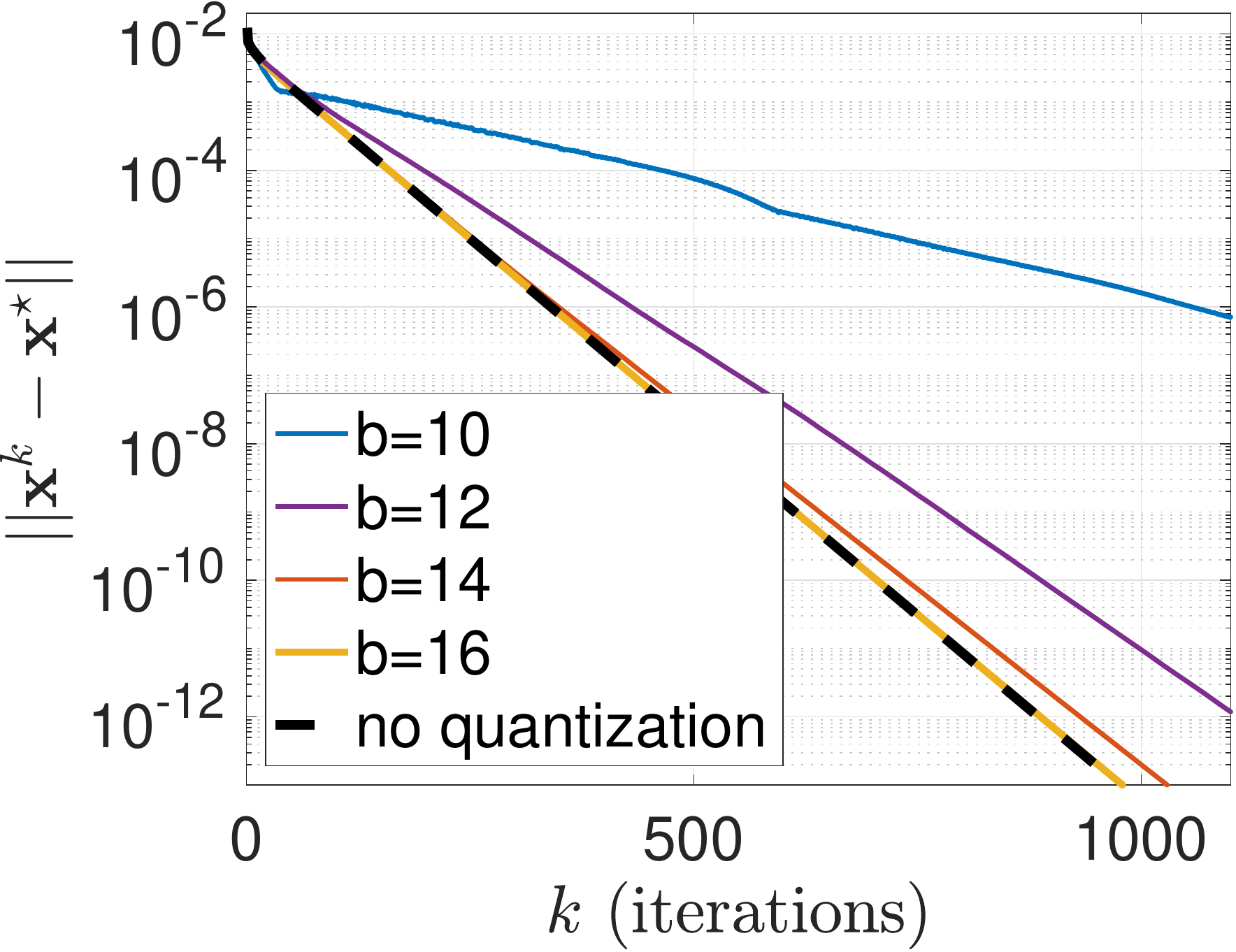}
        \caption{\texttt{SP-data} }
        \label{fig:Decb}
    \end{subfigure}
    \begin{subfigure}[t]{0.32\textwidth}
        \includegraphics[width=\textwidth]{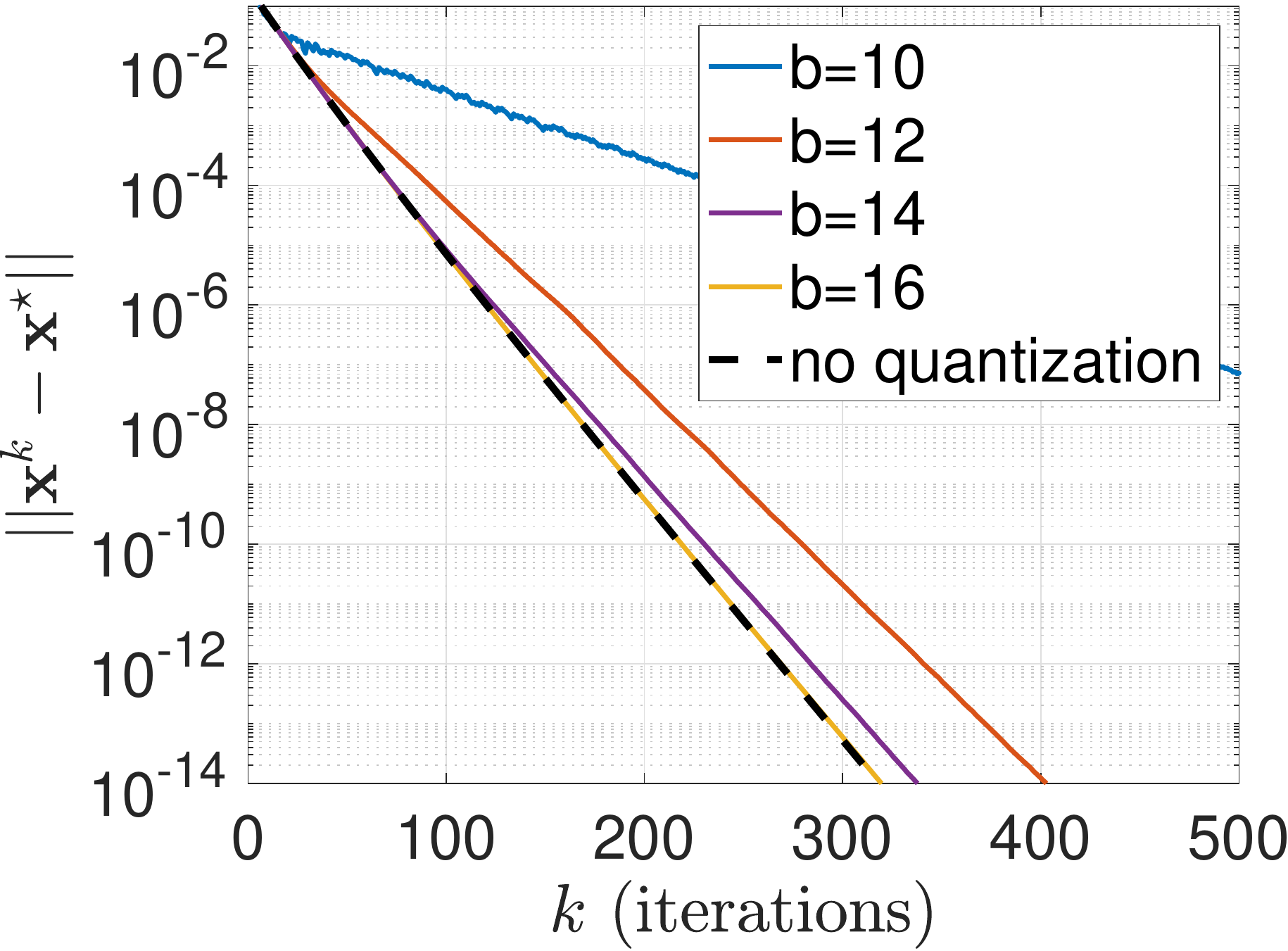}
        \caption{\texttt{MNIST}, digit 9}
        \label{fig:Decd}
    \end{subfigure}
        \begin{subfigure}[t]{0.32\textwidth}
        \includegraphics[width=\textwidth]{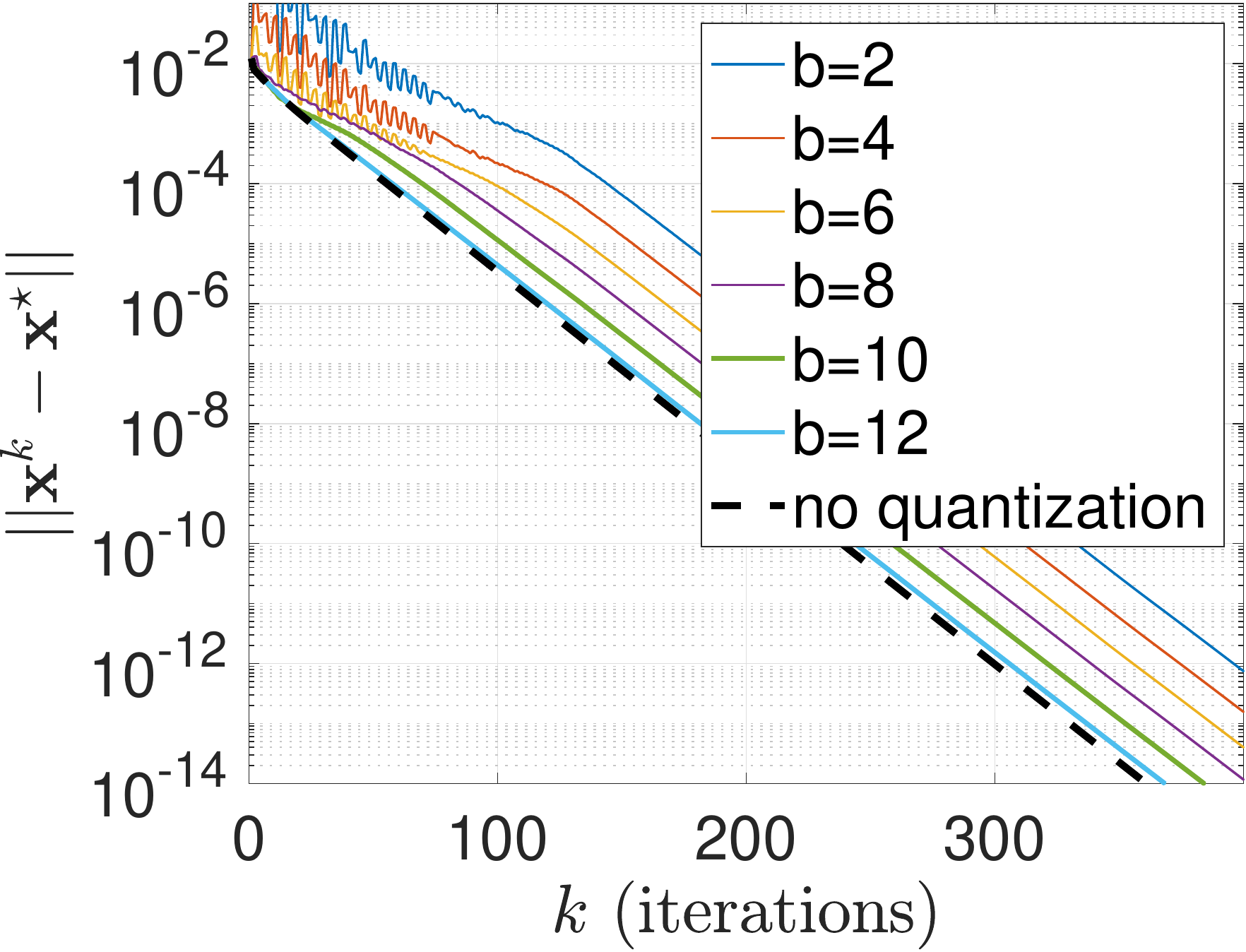}
        \caption{\texttt{SP-data} $\alpha {=} 0.92$, $\gamma{=}0.0002$}
        \label{fig:Decc}
    \end{subfigure}
    \caption{Convergence of the quantized  version of the decentralized algorithm in Section~\ref{sec:ExampleDL}. }
    \label{fig:Dec}
\end{minipage}~
\begin{minipage}{.23\textwidth}
        \includegraphics[width=1\textwidth]{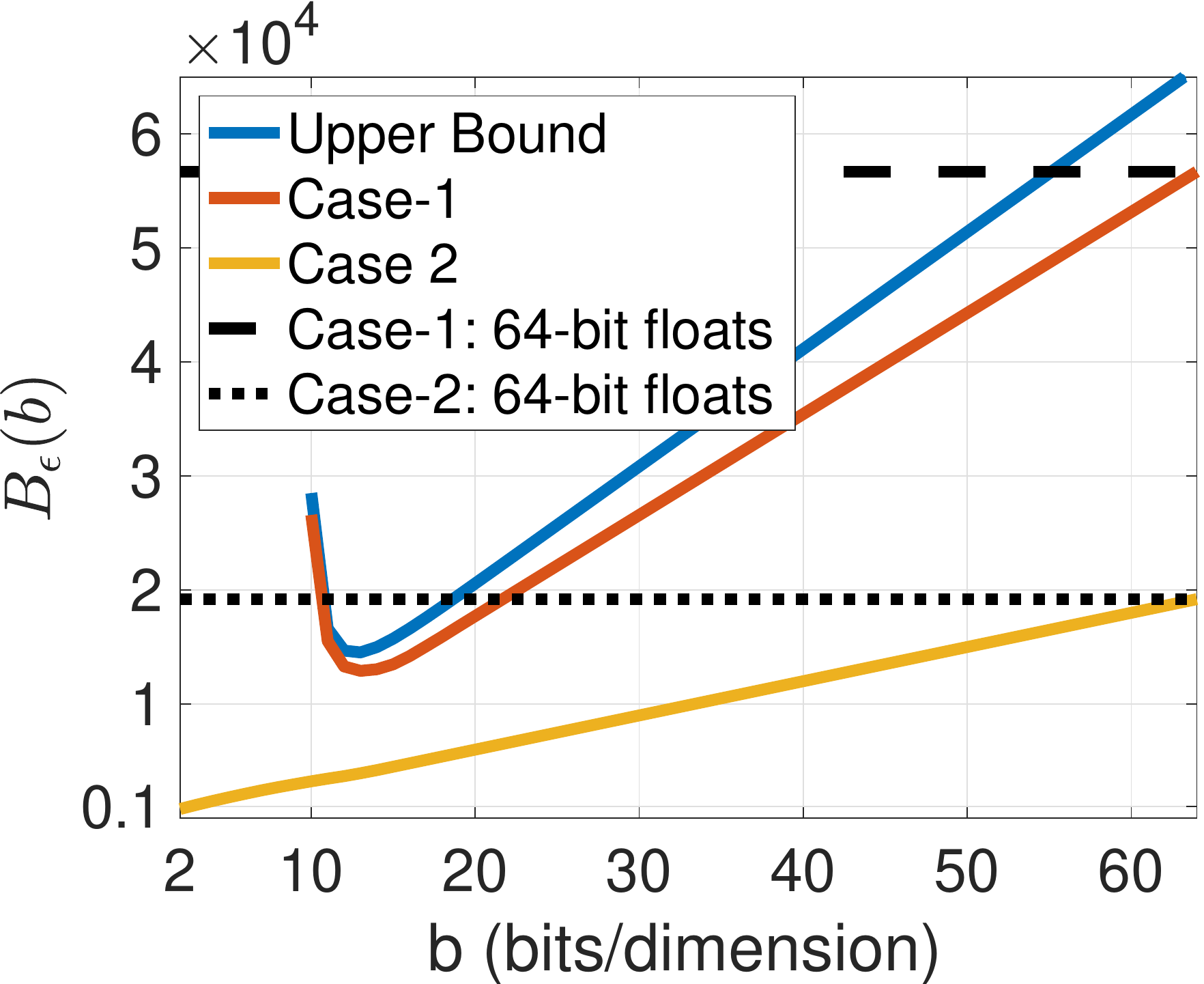}
        \caption{\texttt{SP-data} Bound in Eq.~\eqref{eq:BigB_eps}, $\epsilon=10^{-12}$. }
        \label{fig:Deca}
\end{minipage}
\end{figure*}

\begin{figure*}
\centering
\begin{minipage}{.5\textwidth}
  \centering
    \begin{subfigure}[t]{0.49\textwidth}
        \includegraphics[width=\textwidth]{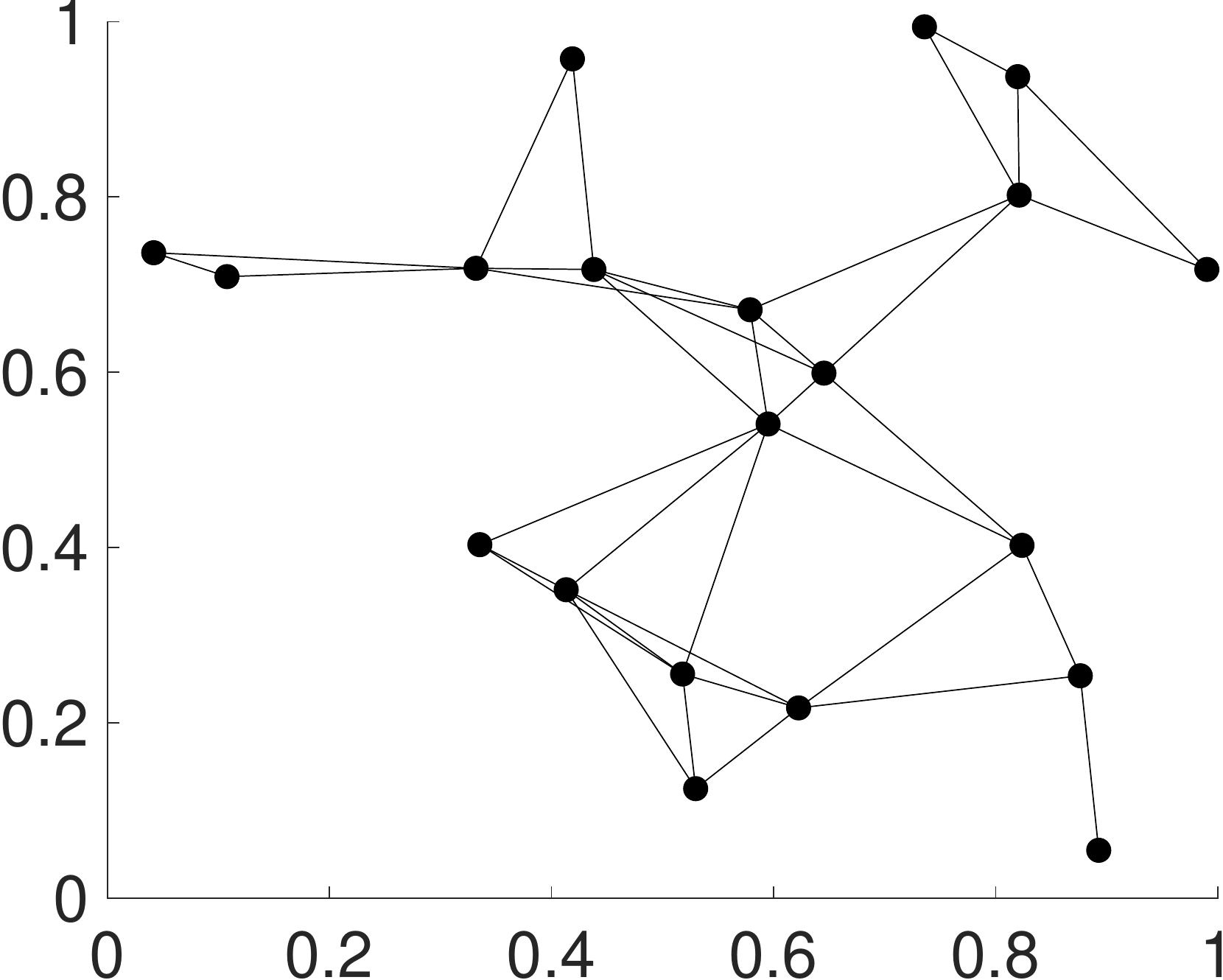}
        \caption{Communication graph.}
        \label{fig:DistA}
    \end{subfigure}~~~
    \begin{subfigure}[t]{0.49\textwidth}
        \includegraphics[width=\textwidth]{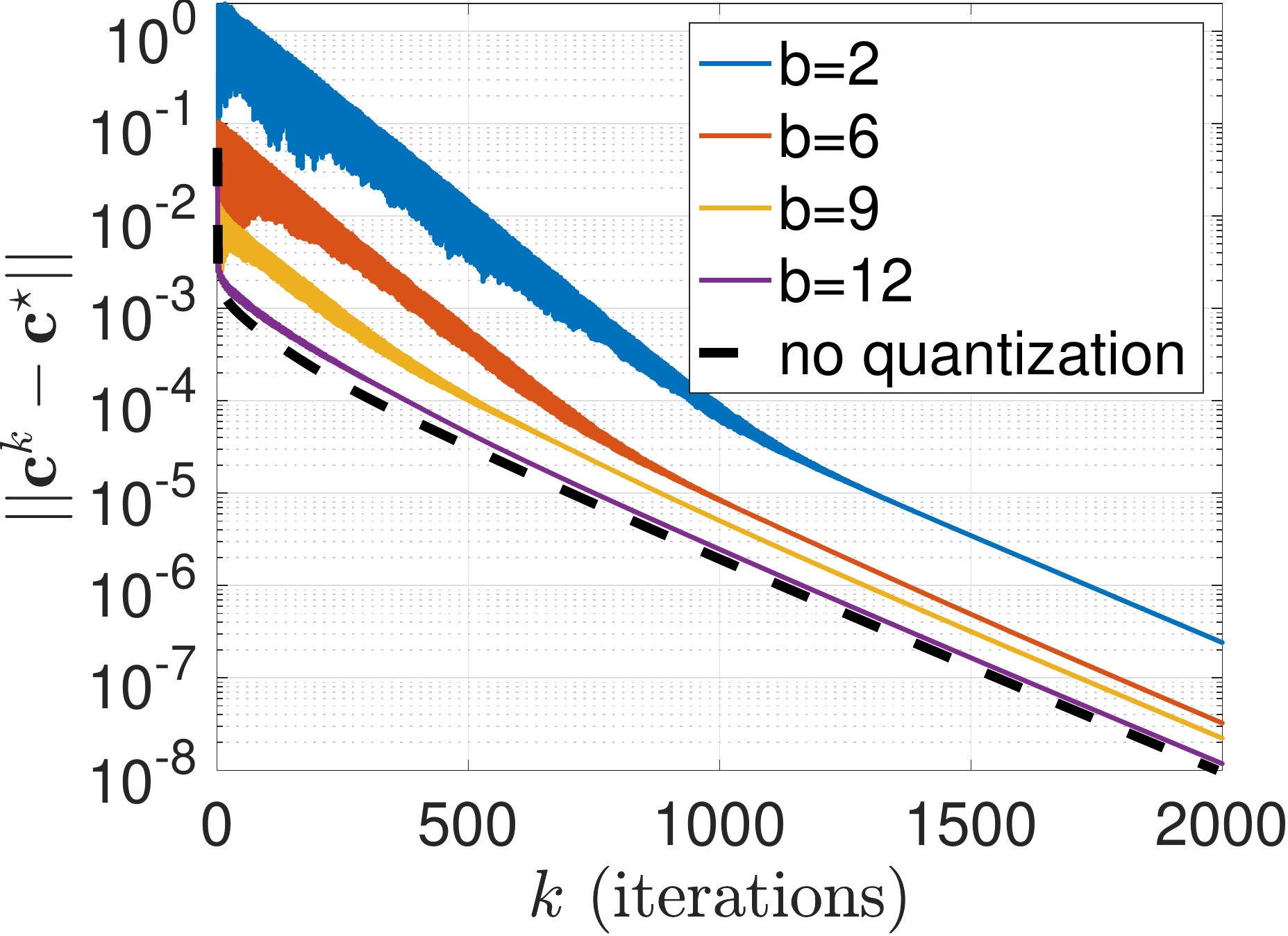}
        \caption{Convergence.}
        \label{fig:DistB}
    \end{subfigure}
    \caption{Distributed algorithm from
    Section~\ref{sec:ExamplesDL-DD}.} \label{fig:Dist} \vspace{-0.3cm}
\end{minipage}~
\begin{minipage}{.49\textwidth}
  \centering
    \begin{subfigure}[t]{0.49\textwidth}
        \includegraphics[width=\textwidth]{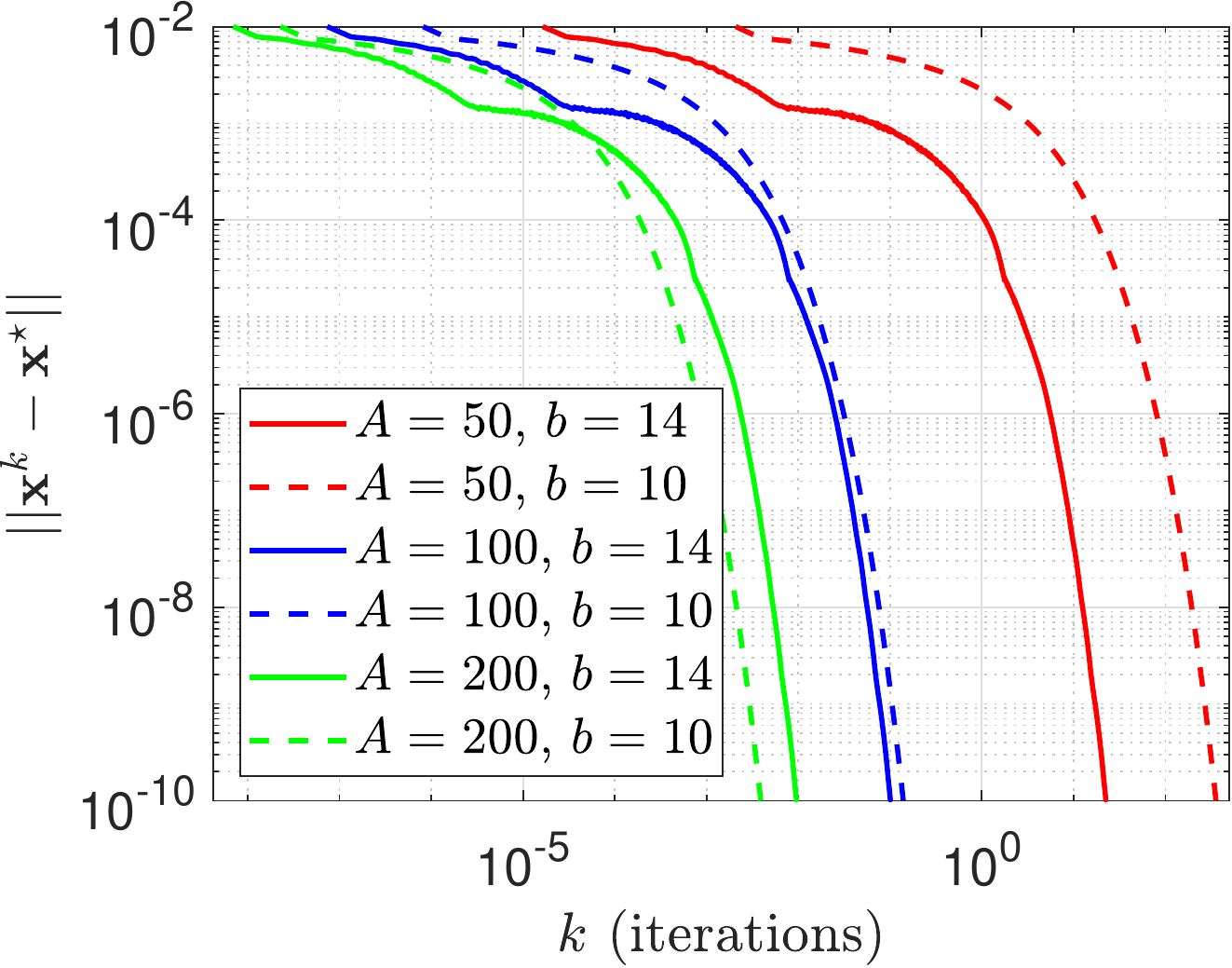}
        \caption{Rate function \textbf{Case 3}.}
        \label{fig:TimeCA}
    \end{subfigure}
    \begin{subfigure}[t]{0.49\textwidth}
        \includegraphics[width=\textwidth]{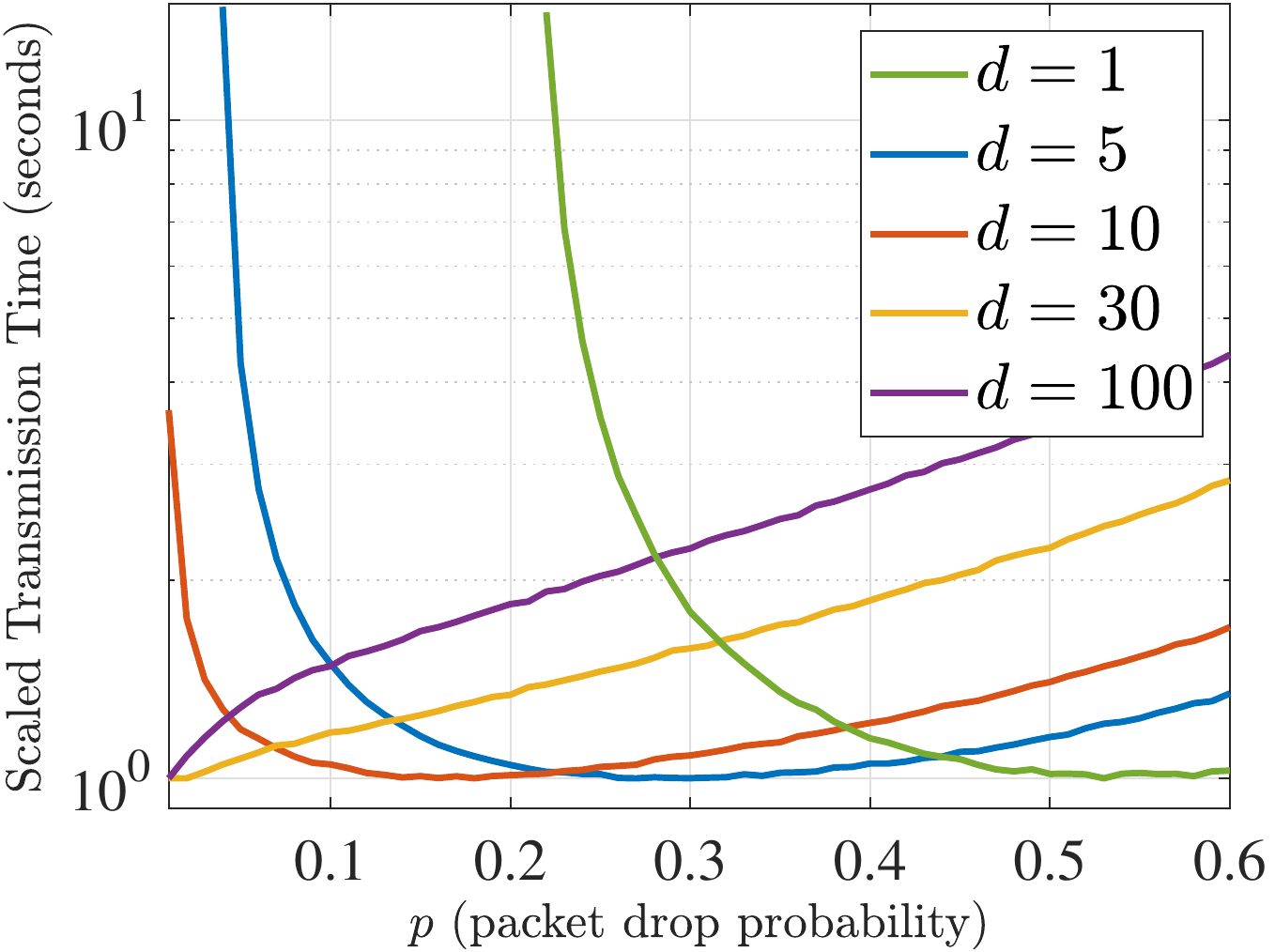}
        \caption{Rate function \textbf{Case 2}.}
        \label{fig:TimeCB}
    \end{subfigure}
    \caption{Transmission time complexity of logistic regression.} \label{fig:TimeC} \vspace{-0.3cm}
\end{minipage}
\end{figure*}

    This convergence can be improved further if we use smaller  $\alpha$ than suggested in Theorem~\ref{MainTheorem} and smaller $\gamma$ than suggested in Proposition~\ref{Theorem:Grad1}. These choices are usually fine in practice since the theorems are generally conservative and derived to capture the worst case behaviour.    By setting $\gamma=2\times 10^{-4}$ and $\alpha=0.92$, Fig.~\ref{fig:rates_cases_2} shows that $b=12$ guarantees  the same convergence rate as the non-quantized algorithm.    More interestingly, our approach allows for an extreme quantization, like $b=2$. With $b=2$ we need to communicate only  $800$ bits per dimension after $400$ iterations to reach $\epsilon=10^{-12}$ accuracy.     To reach the same accuracy without quantizing we need about $300$ iterations so about $19200$ bits per dimension. This means that  our approach reaches the same solution accuracy with $96\%$ less communications overhead.   Note that when $\alpha$ is smaller than suggested in Theorem~\ref{MainTheorem} then it can happen that the grid in Fig.~\ref{fig:Quant} decreases to quickly, so the communicated message $\vec{c}_i^k$ falls outside of the grid. This did not happen for $\alpha= 0.92$ but may happen smaller $\alpha$.  In practice, when the grid decreases too fast ($\alpha$ too small) then the nodes could send a distress signal  (requiring an additional bit) indicating that they need to increase the grid.
We leave such studies for future work.

   Fig.~\ref{fig:Deca} depicts the total number of bits/dimension needed to reach $\epsilon=10^{-12}$ when varying $b$. The figure compares the algorithms when the parameters are selected as in Figure~\ref{fig:Decb} (case 1, red curve),  Figure~\ref{fig:Decc} (case 2, yellow curve), and when no quantization is done (dashed and dotted lines) using 64 bit floating points.  We also plot the upper bound in Eq.~\eqref{eq:BigB_eps} (blue curve) on the total number of bits needed to achieve $\epsilon$-accuracy.  The results show that the upper bound is almost tight. The quantized algorithms always use fewer bits  to reach the $\epsilon$ accuracy. The optimal $b$ are $b=14$ and $b=2$ for the red curve and yellow curve, respectively, communicating only $25\%$ and $4\%$ of the bits needed if there is no quantization. Note that the red curve has a different shape than the yellow curve because in Theorem~\ref{MainTheorem} $\alpha$ is varies depending on $b$, whereas $\alpha=0.92$ is fixed for the yellow curve.

 Fig.~\ref{fig:Dist} shows the convergence of the quantized version of the distributed algorithm in Eq.~\eqref{eq:ItAlg2Grad} where the nodes communicate over a graph with $\alpha=0.99$ and $\gamma=0.7$.
 The communication graph is shown in Fig.~\ref{fig:DistA}. It is randomly generated by distributing the nodes randomly in the box $[0,1]\times[0,1]$ and creating an edge between  nodes if the distance between them is less than $0.3$.
 The results are similar as before, with $b=12$ bits/dimension we get almost the same convergence rate as without quantizing. To achieve $\epsilon=10^{-6}$ accuracy with $b=2$ the quantized algorithm communicates $3466$ bits/dimension during $1733$ iterations but the non-quantized algorithm communicates $71808$ bits during $1122$ iterations. This means that we reduce the communication by $95\%$ to get the $\epsilon$ accuracy compared to the non-quanitzed algorithm.
 We use $\alpha=0.98$

 \subsection{Transmission Time Complexity}

 We now explore the transmission time complexity of the logistic regression algorithm from the previous section using rate functions of \textbf{Case 3} and \textbf{Case 2}, as illustrated in Section~\ref{Sec:Iterplay-R}.
 In these simulations we do not consider the overhead bits, i.e., we set $\theta=0$.

 Fig.~\ref{fig:TimeCA} illustrates the convergence on the \texttt{SP-data} as a function of transmission time when the rate function follows
 $$R(n)= \texttt{Max\_Rate} \times  \frac{n}{A}\exp\left(1 -\frac{n}{A}\right) \frac{\text{bits}}{\text{seconds}},$$
 for some positive constants $\texttt{Max\_Rate}$ and $A$.  The constant $A$ captures the level of channel congestion. Larger $A$ implies that the channel get congested faster (with smaller n). In a slotted-ALOHA or CSMA protocol with $n$ nodes, for example, $A$ is inversely linearly proportional to $n$~\cite{bertsekas1992data,koubaa2006comprehensive}. \footnote{In particular, with the common assumption of Poisson packet arrival for every node and i.i.d. traffics, the additivity of the Poisson distribution implies that we can model the network of $n$ nodes by one virtual node having aggregated arrival of $n$ time~\cite{bertsekas1992data}.} The constant $\texttt{Max\_Rate}$ is the maximum rate, it is achieved by setting $n=A$.
 We set $\texttt{Max\_Rate}=10^{12}$ and $A=50$, $100$, and $200$ and $b=10$ and $14$. From the previous subsection we know that $b=14$ is the optimal quantization size if we want to minimize the total number of communicated bits needed to reach an $\epsilon$-accurate solution.  However, Fig.~\ref{fig:TimeCA} illustrates that as $A$ increases (the channel contention level increases), the network enjoys smaller $b$. In particular, for $A=50$, it takes roughly twice the transmission time to reach the same solution accuracy with $b=10$ as compared to $b=14$. On the other hand, when $A=200$, it takes roughly twice the transmission time to reach the same solution accuracy with $b=14$ as compared to $b=10$. In general, the optimal $b$ decreases as $A$ increases.

 Fig.~\ref{fig:TimeCB} illustrates the convergence for the \textbf{Case 2} transmission-rate function $R_2(\cdot)$ given in Section~\ref{sec:Iterplay-ill}. We explore how the problem dimension $d$ affects the transmission time convergence. To see the affects of $d$, we run logistic regression on a synthetic data set with tuneable $d$: set $y_i=1$ and $\vec{v}_i=\vec{1}+\vec{n}_i\in\R^d$ with probability $0.5$ and $y_i=-1$ and $\vec{v}_i=-\vec{1}+\vec{n}_i\in \R^d$  otherwise where $\vec{n}_i \in \R^d$ are i.i.d. Gaussian variables with zero mean and variance $1$.  We use $\alpha=0.98$ and $b=2$. The figure illustrates the scaled  transmission time needed to reach $\epsilon=10^{-6}$ accuracy as a function of $p$.  In particular, if $\bar{T}_{\epsilon}(p,d)$ is the transmission time until we  find an $\epsilon$ accuracy then the scaled transmission time is $\bar{T}_{\epsilon}(p,d)/\min_{p\in[0,1]} \bar{T}_{\epsilon}(p,d) $.  The results show that for small dimensions, $d=1$, $5$, and $10$, there is a delicate  trade-off between having $p$ too small or too big, however, as $d$ increases it becomes better to choose  small $p$. This means that in the small packet region ($d$ is small) it is good to trade-off packet reliability for quicker messages.  However, as   the packet size increases (with growing $d$) and $R_2(\cdot)$ saturates it becomes better to spend more time to communicating more reliable packets. This can in parts been explained by the fact that when we use large packets  $R_2(\cdot)$ approximates the Shannon capacity and we efficiently achieve more reliable communication.

\section{Conclusions}\label{sec:conclusions}

We investigated the convergence of distributed algorithms under limited communication. We proved that a simple quantization scheme that maps a real-valued vector to a constant number of bits can maintain the linear convergence rate of   unquantized algorithms. We exemplified our results on two classes of communication graphs: \emph{i}) decentralized one where a single master  coordinates information from workers, and \emph{ii}) fully distributed one where nodes coordinate   over a communication  graph. We  numerically  illustrated our theoretical convergence results
in distributed learning on test data including MNIST. 
Our quantization can reduced the communicated overhead (in terms of bits) by $95\%$ in some cases to reach a predefined solution accruacy, compared to using floating points.
 Future work includes exploring adaptive quantization schemes for more general classes of algorithm, including accelerated optimization algorithms and algorithms with sub-linear converge rates.

\appendices

\section{Proof of Theorem~\ref{MainTheorem}} \label{App:Proof_of_main1}

\textbf{Eq.~\eqref{eq:MainBound1}:} We  prove by  mathematical induction that
 $||\vec{x}^{k}-\vec{x}^{\star}|| \leq \alpha(b)^k D$ and
 $||\vec{c}^{k}-\vec{q}^{k}||_{\infty}\leq \tau r^{k-1}$,
 where $\tau=1/(\lfloor 2^{b}\rfloor-1)$.
  For $k=0$ we have $||\vec{x}^{0}-\vec{x}^{\star}||\leq D$ by Assumption~\ref{assumption:contr}-c) and $||\vec{c}^{0}-\vec{q}^{0}|| =0\leq \tau r^{-1}$ since $\vec{q}^0=\vec{c}^0$.
  Suppose that the result holds for some integer $k\geq 0$.  Then
  by the triangle inequality and Eq.~\eqref{eq:AlgLips-1}  we have
 \begin{align*}
     || \vec{x}^{k+1}-\vec{x}^{\star}||
     \leq&  L_A  ||\vec{q}^{k}-\vec{c}^{k}||_{\infty}  +\sigma || \vec{x}^{k}-\vec{x}^{\star}|| \\
    \leq& L_A\tau r^{k-1} + \sigma || \vec{x}^{k}-\vec{x}^{\star}|| \\
        \leq& K\tau \alpha(b)^k D+ \sigma\alpha(b)^k D \\
        =& (K\tau+ \sigma )\alpha(b)^k D
        = \alpha(b)^{k+1}D
\end{align*}
where we have used that $A(\vec{c},\vec{x})$ is $L_A$-Lipschitz continuous in $\vec{c}$ and that $A(C(\vec{x}),\vec{x})$ is $\sigma$-pseudo-contraction in $\vec{x}$.
  We also have by the triangle inequality and Eq.~\eqref{eq:AlgLips-2}  that
  \begin{align*}
     ||\vec{c}^{k+1}-\vec{q}^{k}||_{\infty}  \leq&  || \vec{c}^{k+1}-\vec{c}^{k} ||_{\infty} +||\vec{c}^{k} -\vec{q}^{k}||_{\infty} \\
       \leq&  L_C|| \vec{x}^{k+1}-\vec{x}^{k} || +  r^{k-1} \tau, \\
       \leq&  L_C (|| \vec{x}^{k+1}{-}\vec{x}^{\star} ||{+}|| \vec{x}^{k}{-}\vec{x}^{\star} ||) {+}    r^{k-1} \tau \\
       \leq&  2 L_C \alpha(b)^{k} D +   \tau  \frac{K}{L_A} \alpha(b)^{k}  D, \\
       \leq& K \left( \frac{\sigma}{L_A}+   \tau  \frac{K}{L_A} K^{-1}  \right) \alpha(b)^{k}  D  \\
       \leq& K \left( \frac{\sigma}{L_A}+   \tau  \frac{K}{L_A} \right) \alpha(b)^{k}  D   = r^{k}
  \end{align*}
 Therefore, from Lemma~\ref{lemma:quant}, $||\vec{c}^{k+1}-\vec{q}^{k+1}||_{\infty}\leq \tau r^{k}$.

 \textbf{Prove of Eq.~\eqref{eq:epsacc}:} 
  By using the inequality $1+t\leq \exp(t)$ for all $t\in \R$ we have that
  $$\alpha(b)^k D=(1-(1-\alpha(b)))^kD\leq \exp(-(1-\alpha(b))k)D.$$
  Hence, the result follows by part a) by rearranging $\exp(-(1-\alpha(b))k)D\leq \epsilon$.

  \section{Proof of Corollary~\ref{Corr:TL1} and~\ref{Corr:TL2}} \label{App:CorrProofs}

    \textbf{Corollary~\ref{Corr:TL1}:} From Proposition~\ref{Theorem:Grad1} and Theorem~\ref{MainTheorem} we have $\sigma=(\kappa-1)/(\kappa+1)$, $L_A=2\sqrt{d}/(\mu+L)$, $L_C=L$ and
   $$K\leq\frac{2L_A L_C}{\sigma}\leq 12 \sqrt{d},$$
   since $L/(\mu+L)\leq 1$ and $\kappa\geq 2$.  Therefore, using that $2^b-1\geq 2^{b-1}$, we have
   $\alpha(b)\leq 1-({2}/{\kappa+1})+ {24 \sqrt{d}}/{2^b},$
   or by rearranging
  $$\frac{1}{1-\alpha(b)} \leq \frac{\kappa+1}{2} \frac{1}{1-12\sqrt{d}(\kappa+1)/2^b} \leq \kappa+1, $$
  where we have used that $b= \lceil \log_2(24\sqrt{d}(\kappa+1)) \rceil$.
   Eq.~\eqref{eq:epsacc} in Theorem~\ref{MainTheorem}  now yields
     $$ \log_2\left(24(\kappa+1) \sqrt{d} \right) (\kappa+1) d \left(\log \left(D\right)+\log \left(\frac{1}{\epsilon}\right)\right) \frac{\text{ bits}}{\text{node}}.$$

     \textbf{Corollary~\ref{Corr:TL2}:} From~\cite{nesterov2013introductory}, the algorithm is $\sigma$-linear with
$$\sigma=\sqrt{1-\frac{1}{\kappa}}\leq 1-\frac{1}{2\kappa}.$$
We also have $L_A=\sqrt{d}/L$, $L_C=L$ meaning that
$K\leq 2\sqrt{2d}$,
$$\alpha(b)\leq 1-\frac{1}{2\kappa} + \frac{4\sqrt{2d}}{2^b},$$
and
$$\frac{1}{1-\alpha(b)}=\frac{2\kappa}{1-8\kappa \sqrt{2d}/2^b} \leq 4\kappa,$$
since $b=\left\lceil  \log_2(16 \kappa \sqrt{2d} )\right\rceil$.   Eq.~\eqref{eq:epsacc}  now yields
 $$4 \left\lceil  \log_2(16 \kappa \sqrt{2d} )\right\rceil \kappa d \left(\log \left(d\right)+\log \left(\frac{1}{\epsilon}\right)\right) \frac{\text{ bits}}{\text{node}}.$$

  \section{Proof of Theorem~\ref{Thm:CUS}} \label{App:ProofTP}

Let $M$ be the random variable defined in Eq.~\eqref{eq:mk_RV}.
Since   $m_i$ are all realizations of $M$ we have
     $E[T_\epsilon(b,\theta,p)]  
                                             =\Delta(b,\theta,p) \times  k_{\epsilon}(b) \times  E[M].$
 To finish the prove we derive the inequality
   $$\frac{\log((N-1)N)}{\log(1/p)}\leq E[m] \leq  1+ \frac{1}{\log(1/p)}+\frac{\log((N-1)N)}{\log(1/p)}.$$
Recall that
    $P\left[ S_l =m\right]=p^{m-1}(1-p)$ and 
    $P\left[S_l\leq m\right]=1-p^{m}$.
 Using that  the communications are independent across  links $\mathcal{L}$  we have
   $P(M\leq m)=(1-p^{m})^{|\mathcal{L}|}.$
%
  This means that
  \begin{align*}
      E[M]=& \sum_{m=1}^{\infty} m P(M=m) =  \sum_{m=0}^{\infty} P(M> m) \\
                 =&  \sum_{m=1}^{\infty} (1- P(M\leq m)) = \sum_{m=1}^{\infty} (1- (1-p^{m})^{|\mathcal{L}|}).
  \end{align*}
 The last sum can be bounded by integral as follows
 \begin{align*}
     \int_0^{\infty} (1-h(z)^{|\mathcal{L}|}) dz \leq& \sum_{m=1}^{\infty} (1- h(m)^{|\mathcal{L}|}) \\ \leq& 1+  \int_0^{\infty} (1-h(z)^{|\mathcal{L}|}) dz
 \end{align*}
 where  $h(z)=1-p^z$ and we have used that $(1-h(z)^{|\mathcal{L}|})\geq0$ is decreasing in $z$.
 Using that $h'(z)/(1-h(z))=-\log(p)$, integration by substitution gives
 \begin{align*}
    \int_0^{\infty} (1{-} h(z)^{|\mathcal{L}|} )dz=&  \frac{-1}{\log(p)}\int_0^{\infty} \frac{1- h(z)^{|\mathcal{L}|}}{1-h(z)} h'(z) dz \\
          =&  \frac{-1}{\log(p)}\int_0^{1} \frac{1- w^{|\mathcal{L}|}}{1-w} dw \\
          =&  \frac{-1}{\log(p)}\int_0^{1} \sum_{i=0}^{|\mathcal{L}|-1}  w^i dw
          {=}  \frac{-1}{\log(p)} \sum_{i=1}^{|\mathcal{L}|} \frac{1}{i}.
 \end{align*}
  The result now follows from the fact that
 $$   \log(|\mathcal{L}|) \leq \sum_{i=1}^{|\mathcal{L}|} \frac{1}{i}\leq 1+ \log(|\mathcal{L}|).$$

\section{Proof of Theorem~\ref{th4}} \label{App:ProofTF}

 Using the independence of failure events in all   links $\mathcal{L}$  we have
   $P(M\leq m)=(1-p^{m})^{|\mathcal{L}|}.$
   From Theorem~\ref{MainTheorem}, we find an $\epsilon$-accurate solution  after  $k\geq k_{\epsilon}(b)$ successful iterations.
   Therefore, we just have to ensure that the communication of the first $k$ iterations is successful.
   In particular, we need
   \begin{align*}
        P(M_1\leq m,\ldots,M_k\leq m)
                    =& P(M_1\leq m)  {\cdots} P(M_k\leq m)\\
                     =&(1-p^{m})^{ |\mathcal{L}|k} \geq \delta
   \end{align*}
   where $k=k_{\epsilon}(b)$. By rearranging we get that
   $$m \geq \log_p (1-\delta^{1/(|\mathcal{L}|k)})= \frac{\log\left(1-\delta^{1/(|\mathcal{L}|k)}\right) }{\log(p)}. $$
 Using that $\log(y+1)\geq y/(y+1)$ for all $y>-1$ we have
 \begin{align*}
    \frac{\log\left(1{-}\delta^{1/(|\mathcal{L}|k)}\right) }{\log(p)}&= \frac{{-}\log\left(1-\delta^{1/(|\mathcal{L}|k)}\right) }{\log(1/p)}, \\
     &\leq \frac{1}{1{-}\delta^{1/(|\mathcal{L}|k)}} \frac{1}{\log(1/p)},
 \end{align*}
 where the last inequality is due to  $\delta^{1/(|\mathcal{L}|k)}\leq 1$. 
 To finish the proof, we should show that  
 $$  \frac{1}{1{-}\delta^{1/(|\mathcal{L}|k)}} \leq \frac{|\mathcal{L}|k}{1-\delta}.$$
This inequality can be derived by using the following variant of Bernoulli's inequality: $(1+x)^r \leq 1+rx$ for all $r\in[0,1]$ and $x\geq -1$. By applying the change of variables $\delta=1+x$ and $r=1/(|\mathcal{L}|k)$ and rearranging we obtain the result.

\section{Proof of Proposition~2}
 \label{Appendix:Proof_Prop2}

 The algorithm in Eq.~\eqref{eq:ItAlg2Grad} in Section~\ref{sec:ExamplesDL-DD} is a variant of dual decomposition and is similar to the algorithm~\cite{uribe2018dual}.
 To derive the algorithm we reformulate the optimization problem in Eq.~(4) in consensus form as follows:
\begin{equation} \label{eq:mainOptProb-dual}
\begin{aligned}
& \underset{x}{\text{minimize}}
& &   F(\vec{c}):=\sum_{i=1}^N f_i(\vec{c}_i)  \\
&  \text{subject to} && \bec{W}\vec{c}=\vec{0} \iff \vec{A}\vec{c}=\vec{0}
\end{aligned}
\end{equation}
 where $\bec{W}=\vec{W}\otimes \vec{I}$ and $\vec{A}=\sqrt{\bec{W}}$ is the matrix square root of $\bec{W}$, i.e., $\vec{A}\tran\vec{A}=\bec{W}$, obtained by eigenvalue decomposition.\footnote{$\bec{W}$ and $\vec{A}$ have the same null-space so the constraints $\bec{W}\vec{x}=\vec{0}$ and $\vec{A}\vec{x}=\vec{0}$ are equivalent.}
 In particular, since $\bec{W}$ is symmetric and semi-positive-definite we can write $\vec{A}=\sqrt{\BLa}\vec{Q}$ where  $\bec{W}=\vec{Q}\tran \BLa \vec{Q}$ and the columns of $\vec{Q}$ are the normalized eigenvectors of $\bec{W}$, and  $\BLa$ is a diagonal matrix where the diagonal elements are the eigenvalues of $\bec{W}$ in descending order, i.e.,
 $$\BLa=\texttt{diag}\left(\underbrace{\lambda_{\max}(\bec{W}),\ldots,\lambda_{\min}^+(\bec{W})}_{d(N-1) \text{ components}}, \underbrace{0,\ldots,0}_{d \text{ components}}\right).$$
 Note that because the last $d$ diagonal elements of $\BLa$ are zero the last $d$ rows of a $\vec{A}$ are zero.
 Define $\bec{A}\in \R^{d(N-1)\times d}$ as  the non-zero rows of $\vec{A}$.
  $ \bec{A}$ has full row rank since the eigenvectors are linearly independent, so the square matrix $\bec{A} \bec{A}\tran$ is non-singular. 
 We show that the algorithm function $\vec{x}\mapsto A(C(\vec{x}),\vec{x})$ is contractive on $\mathcal{X}:=\texttt{Im}(\bec{W})$ in the following norm.
\begin{lemma}
   Define $\vec{M}:=(\bec{A}\bec{A}\tran)^{-1}\bec{A}$. Then
   $||\vec{x}||_{\vec{M}}:= ||\vec{M} \vec{x}||_2$
   is a norm on $\mathcal{X}$.\footnote{Note that $||\cdot||_{\vec{M}}$ is not a norm on all of $\R^{Nd}$.}
   Moreover, if we define $M_1=||\vec{M}||_2$ and $M_2=||\vec{A}\tran||_2$ then
   $ ||\vec{x}||_{\vec{M}} \leq M_1 ||\vec{x}||_2$  and $||\vec{x}||_2 \leq M_2 ||\vec{x}||_{\vec{M}}$,
   for all $\vec{x}\in \mathcal{X}$.
\end{lemma}
\begin{proof}
   To prove that $||\vec{x}||_{\vec{M}}$ is a norm on $\mathcal{X}$, we need to show that following holds for all $\vec{x},\vec{y}\in \mathcal{X}$~\cite[Definition 5.1.1]{horn2013matrix}:
     (1a) $||\vec{x}||_{\vec{M}}\geq0$,
     (1b) $||\vec{x}||_{\vec{M}}=0$ if and only if $\vec{x}=0$,
    (2) $||c\vec{x}||_{\vec{M}}=c ||\vec{x}||_{\vec{M}}$ for any positive scalar $c$, and
     (3) $||\vec{x}+\vec{y}||_{\vec{M}}\leq  ||\vec{x}||_{\vec{M}}+||\vec{y}||_{\vec{M}}$.

     Conditions (1a), (2), and (3) follow directly from the fact that $||\cdot||_2$ is a norm.
     To show that (1b) holds true note that if $\vec{x}\in \mathcal{X}=\texttt{Im}(\bec{W})$ then there exists $\vec{z}\in \R^{dN}$ such that $\vec{x}= \bec{W}\vec{z}$. 
    Moreover, by using that $\bec{W}=\vec{A}\tran \vec{A}$ and the definition of $\bec{A}$ we also have that $\vec{x}=\bec{A}\tran \bec{A} \vec{z}$ and $||\vec{x}||_{\vec{M}}= \bec{A} \vec{z}$.
    Therefore, if $\vec{x}\neq \vec{0}$ then $||\vec{x}||_{\vec{M}}= ||\bec{A} \vec{z}||_2\neq 0$, since $\bec{A} \vec{z}=\vec{0}$ implies that $\vec{x}=\vec{0}$, which proves that condition (1b) holds true.

    The inequality $ ||\vec{x}||_{\vec{M}} \leq M_1 ||\vec{x}||_2$ follows directly from the definition of the norm $||\cdot||_{\vec{M}}$.
    The inequality $||\vec{x}||_2 \leq M_2 ||\vec{x}||_{\vec{M}}$ can be obtained similarly by noting that $\vec{x}=\vec{A}\tran \vec{M} \vec{x}$.
\end{proof}

 To prove the contraction, we show that the algorithm is a standard dual decomposition algorithm after some change of variables.
 The dual function of the problem in Eq.~\eqref{eq:mainOptProb-dual} is
 $$ D(\vec{v})= \min_{\vec{c}}  F(\vec{c}) +\vec{v}\tran \vec{A}\vec{c} = \vec{c}(\vec{A}\tran \vec{v})$$
 where 
$\vec{c}(\vec{x}) =  \text{argmin}_{\vec{c}} ~F(\vec{c}) +\langle \vec{x}, \vec{c}\rangle. $
 The dual gradient is
 $\nabla D(\vec{v})= 
  \vec{A} \vec{c}(\vec{A}\tran\vec{v})$
 and the dual function can be maximized by the gradient method 
 \begin{align*}
    \vec{v}^{k+1}=&  \vec{v}^k+\gamma  \vec{A} \vec{c}^k, \\
     \vec{c}^{k+1}=& \underset{\vec{c}}{\text{argmin}} ~F(\vec{c}) +\langle  \vec{A}\tran \vec{v}^{k+1}, \vec{c}\rangle.
 \end{align*}
 By doing the change of variables $\vec{x}=\vec{A}\tran \vec{v}$ and multiplying the $\vec{v}$-update by $\vec{A}\tran$ the gradient method can be written on the following form
 \begin{align*}
    \vec{x}^{k+1}=& ~ \vec{x}^k+\gamma  \bec{W} \vec{c}^k, \\
     \vec{c}^{k+1}=&~ \underset{\vec{c}}{\text{argmin}} ~F(\vec{c}) +\langle\vec{x}^{k+1}, \vec{c}\rangle,
 \end{align*}
 which is the same as the algorithm in Eq.~(7) in Section 2.2.2.
 Note that if $\vec{v}^0\in\mathcal{V}:=\texttt{Im}(\vec{A})$ and $\vec{x}^{0}\in \mathcal{X}$ then all the iterates are in $\mathcal{V}$ and $\mathcal{X}$, respectively, i.e., $\vec{v}^{k}\in\mathcal{V} $ and $\vec{x}^{k}\in\mathcal{X} $ for all $k\in \N$.
 We can go from $\vec{x}=\vec{A}\tran \vec{v}\in \mathcal{X}$  to $\vec{v}\in \mathcal{V}$ by the transform
 $\vec{v}=\vec{M}\vec{x}=(\bec{A}\bec{A}\tran)^{-1}\bec{A} \vec{x}.$
 We now conclude the prove by showing that the function
 \begin{equation} \label{eq:G_trans}G:\mathcal{V}\rightarrow \mathcal{V}, ~~\vec{v}\mapsto \vec{v}+\gamma \vec{A} \vec{c}(\vec{A}\tran \vec{v})\end{equation}
 is a contraction in the 2-norm and that   $\vec{x} \mapsto A(C(\vec{x}),\vec{x})$ is a contraction on $\mathcal{X}$ in the norm $||\cdot||_{\vec{M}}$.
 \begin{lemma}
    Suppose that the step-size is chosen as
    $$\gamma\in\left(0,\frac{2L\mu}{\mu \lambda_{\min}^+(\vec{W})+L\lambda_{\max}(\vec{W})}\right].$$
    Then for all $\vec{v}\in \mathcal{V}$ and $\vec{x}\in \mathcal{X}$ we have
   \begin{align*}
        ||G(\vec{v})-\vec{v}^{\star}||_2 \leq&\sigma_{\gamma} ||\vec{v}-\vec{v}^{\star}||_2, \\
         ||A(C(\vec{x}),\vec{x})-\vec{x}^{\star}||_{\vec{M}} \leq&\sigma_{\gamma} ||\vec{x}-\vec{x}^{\star}||_{\vec{M}},
    \end{align*}
    where $G(\cdot)$ is as defined in Eq.~\eqref{eq:G_trans} and
    $$\sigma_{\gamma}:=   \sqrt{1-\frac{2\gamma \lambda_{\min}^+(\vec{W}) \lambda_{\max}(\vec{W})  }{ \mu\lambda_{\min}^+(\vec{W})+L\lambda_{\max}(\vec{W})} },$$
    and $\vec{v}^{\star}$ is a maximizer of the dual function $D(\cdot)$ and $\vec{x}^{\star}=\vec{A}\tran \vec{v}^{\star}$.
    In particular, if we choose 
    $\gamma= 2L\mu/(\mu\lambda_{\min}^+(\vec{W})+L\lambda_{\max}(\vec{W}))$
    then for any $\vec{v}\in \mathcal{V}$ we have
    \begin{align*}
         ||G(\vec{v})-\vec{v}^{\star}||_2 \leq& \sigma ||\vec{v}-\vec{v}^{\star}||_2, \\
        ||A(C(\vec{x}),\vec{x})-\vec{x}^{\star}||_{\vec{M}} \leq& \sigma ||\vec{x}-\vec{x}^{\star}||_{\vec{M}},
    \end{align*}
    where $\sigma=1-2/(\kappa+1)$ and $\kappa=\lambda_{\max}(\vec{W}) L/ (\mu \lambda_{\min}^+(\vec{W}))$. %
 \end{lemma}
 \begin{proof}
 The iterative algorithm defined by $G(\cdot)$ is a gradient ascent  for the concave dual function $D(\cdot)$, which is $\lambda_{\min}^+(\vec{W})/L$-strongly concave and $\lambda_{\max}(\vec{W})/\mu$-smooth (we show this in Lemma~\ref{lemma:D} below).  Therefore, by standard results in convex optimization, $G(\cdot)$ is contractive on $\mathcal{V}$ with the contractivity parameter $\sigma_{\gamma}$, see e.g.~\cite[Theorem 2.1.15]{nesterov2013introductory}.
 To prove that $A(C(\vec{x}),\vec{x})$ is a contraction on $\mathcal{X}$, take $\vec{x}\in \mathcal{X}$ and $\vec{M}\vec{x}\in \mathcal{V}$.  Then by setting $\vec{x}^+=A(C(\vec{x}),\vec{x})$ and $\vec{v}^+=G(\vec{v})$, we have
\begin{align*}
  ||\vec{x}^+{-}\vec{x}^{\star}||_{\vec{M}}= ||\vec{v}^+{-}\vec{v}^{\star}||_2\leq \sigma_{\gamma}  ||\vec{v}{-}\vec{v}^{\star}||_2 = \sigma ||\vec{x}{-}\vec{x}^{\star}||_{\vec{M}},
\end{align*}
which yields the results.
 \end{proof}

  \begin{lemma} \label{lemma:D}
     $D(\cdot)$ is a) $\lambda_{\min}^+(\vec{W})/L$-strongly concave on the set $\mathcal{V}$ and $\lambda_{\max}(\vec{W})/\mu$-smooth.
  \end{lemma}
\begin{proof} 
  We can write
  $D(\vec{v})=-F^*(-\vec{A}\tran \vec{v})$ and $\nabla D(\vec{v})=  \vec{A} \nabla F^*(-\vec{A}\tran \vec{v})$ 
  where $F^*(\cdot)$ is the  convex conjugate of $F(\cdot)$.
  To prove the strongly concavity we use that $F^*(\cdot)$ is $1/L$-strongly convex since $F(\cdot)$ is $L$-smooth, see~\cite[Proposition 12.60]{rockafellar2009variational}. Take
  $\vec{v}_1,\vec{v}_2\in  \mathcal{V}$, then
  \begin{align*}
     \langle \nabla D(\vec{v}_2)-&\nabla D(\vec{v}_1),\vec{v}_1-\vec{v}_2 \rangle \\
      =& \langle F^*(\vec{A}\tran \vec{v}_1)-F^*(\vec{A}\tran\vec{y}_2),\vec{v}_1-\vec{v}_2 \rangle \\
   =& \langle F^*(\vec{A}\tran \vec{v}_1)-F^*(\vec{A}\tran\vec{v}_2),\vec{A}\tran\vec{v}_1-\vec{A}\tran\vec{v}_2) \rangle    \\
   =& \frac{1}{L} ||\vec{A}\tran (\vec{v}_1-\vec{v}_2)||^2
   \geq \frac{\lambda_{\min}^+(\vec{W})}{L} ||\vec{v}_1-\vec{v}_2||^2
  \end{align*}
  where we have used that   $\vec{y}_1,\vec{y}_2\in  \mathcal{V}$ and $\lambda_{\min}^+(\vec{A} \vec{A}\tran)=\lambda_{\min}^+(\vec{W})$ to get the last inequality.

  To prove the smoothness of $D(\cdot)$  we use that $F^*(\cdot)$ is $1/\mu$-strongly convex.
   Take some $\vec{v}_1,\vec{v}_2$,   then
  \begin{align*}
    ||\nabla D(\vec{v}_1)-\nabla D(\vec{v}_2)||_2 \leq& ||\vec{A}|| ||\nabla F^*(\vec{A}\tran \vec{v}_1)-\nabla F^*(\vec{A}\tran \vec{v}_2)||_2\\
            \leq& \frac{ ||\vec{A}||  ||\vec{A}\tran||}{\mu}   ||\vec{v}_1-\vec{v}_2||_2  \\
             \leq& \frac{ \lambda_{\max}(\vec{W})}{\mu}   ||\vec{v}_1-\vec{v}_2||_2,
  \end{align*}
  which concludes the proof.
\end{proof}

\bibliographystyle{IEEEtran}
\bibliography{refs}{}

\end{document}